\documentclass[12pt]{amsart}

\usepackage{amsthm,amsfonts,amsmath,amssymb,latexsym,epsfig,mathrsfs,yfonts,marvosym}
\usepackage[usenames]{color}
\usepackage{graphicx}
\usepackage[all]{xy}
\usepackage{epsfig}
\usepackage{epic}
\usepackage{eepic}
\usepackage{hyperref}
\usepackage{dsfont}\let\mathbb\mathds
\usepackage[english]{babel}

\DeclareMathAlphabet\oldmathcal{OMS}        {cmsy}{b}{n}
\SetMathAlphabet    \oldmathcal{normal}{OMS}{cmsy}{m}{n}
\DeclareMathAlphabet\oldmathbcal{OMS}       {cmsy}{b}{n}
 
\usepackage{eucal}

\usepackage[active]{srcltx}

\newtheorem{theorem}{Theorem}[section]
\newtheorem{lemma}[theorem]{Lemma}
\newtheorem{proposition}[theorem]{Proposition}
\newtheorem{corollary}[theorem]{Corollary}
\newtheorem{definition}[theorem]{Definition}

\newtheorem{def/prop}[theorem]{Definition/Proposition}
\newtheorem{algorithm}[theorem]{Algorithm}

\newenvironment{example}{\medskip \refstepcounter{theorem}
\noindent  {\bf Example \thetheorem}.\rm}{\,}
\newenvironment{remark}{\medskip \refstepcounter{theorem}
\newcommand     {\comment}[1]   {}
\newcommand{\mute}[2] {}
\newcommand     {\printname}[1] {}

\noindent  {\bf Remark \thetheorem}.\rm}{\,}
\def\<{\langle}

\def\>{\rangle}

\def\BOne{{\mathchoice {\rm 1\mskip-4mu l} {\rm 1\mskip-4mu l}
                          {\rm 1\mskip-4.5mu l} {\rm 1\mskip-5mu l}}}

\def\fract#1#2{\raise4pt\hbox{$ #1 \atop #2 $}}
\def\decdnar#1{\phantom{\hbox{$\scriptstyle{#1}$}}
\left\downarrow\vbox{\vskip15pt\hbox{$\scriptstyle{#1}$}}\right.}

\def\bbc{{\mathbb C}}

\def\bbp{{\mathbb P}}
\def\bbq{{\mathbb Q}}
\def\bbr{{\mathbb R}}

\def\bbz{{\mathbb Z}}

\def\gra{\alpha}
\def\grb{\beta}

\def\gre{\epsilon}

\def\grg{\gamma}

\def\grk{\kappa}
\def\grl{\lambda}

\def\gro{\omega}

\def\grr{\rho}

\def\grt{\tau}

\def\grD{\Delta}

\def\grG{\Gamma}

\def\grL{\Lambda}
\def\grO{\Omega}

\def\grS{\Sigma}

\def\bfl{{\bf l}}

\def\bfr{{\bf r}}

\def\bfv{{\bf v}}
\def\bfw{{\bf w}}

\def\cala{{\mathcal A}}

\def\cald{{\mathcal D}}

\def\calf{{\mathcal F}}

\def\cali{{\mathcal I}}

\def\calp{{\mathcal P}}

\def\cals{{\oldmathcal S}}

\def\calw{{\mathcal W}}

\def\la#1{\hbox to #1pc{\leftarrowfill}}
\def\ra#1{\hbox to #1pc{\rightarrowfill}}

\def\ga{{\mathfrak a}}

\def\gp{{\mathfrak p}}

\def\gr{{\mathfrak r}}

\def\gt{{\mathfrak t}}
\def\gu{{\mathfrak u}}

\def\gz{{\mathfrak z}}

\def\gA{{\mathfrak A}}

\def\gR{{\mathfrak R}}

\def\gtz{\tilde{\mathfrak z}}

\def\hook{\mathbin{\hbox to 6pt{%
                 \vrule height0.4pt width5pt depth0pt
                 \kern-.4pt
                 \vrule height6pt width0.4pt depth0pt\hss}}}

\def\tTheta{\tilde{\Theta}}
\def\ttheta{\tilde{\theta}}

\def\th{\tilde{h}}
\def\tU{\tilde{U}}

\def\12{\xi_{k_1,k_2}}
\def\m5{M^5_{k_1,k_2}}

\begin{document}

\title{The Sasaki Join, Hamiltonian 2-forms, and Constant Scalar Curvature}

\author{Charles P. Boyer and Christina W. T{\o}nnesen-Friedman}\thanks{Both authors were partially supported by grants from the Simons Foundation, CPB by (\#245002) and CWT-F by (\#208799)}
\address{Charles P. Boyer, Department of Mathematics and Statistics,
University of New Mexico, Albuquerque, NM 87131.}
\email{cboyer@math.unm.edu} 
\address{Christina W. T{\o}nnesen-Friedman, Department of Mathematics, Union
College, Schenectady, New York 12308, USA } \email{tonnesec@union.edu}



\maketitle

\markboth{Sasaki Join, Hamiltonian 2-forms, CSC Metrics}{Charles P. Boyer and Christina W. T{\o}nnesen-Friedman}

\begin{abstract}
We describe a general procedure for constructing new explicit Sasaki metrics of constant scalar curvature (CSC), including Sasaki-Einstein metrics, from old ones. We begin by taking the join of a regular Sasaki manifold of dimension $2n+1$ and constant scalar curvature with a weighted Sasakian 3-sphere. Then by deforming in the Sasaki cone we obtain CSC Sasaki metrics on compact Sasaki manifolds $M_{l_1,l_2,\bfw}$ of dimension $2n+3$ which depend on four integer parameters $l_1,l_2,w_1,w_2$. Most of the CSC Sasaki metrics are irregular. We give examples which show that the CSC rays are often not unique on the underlying fixed strictly pseudoconvex CR manifold. Moreover, it is shown that when the first Chern class of the contact bundle vanishes, there is a two dimensional subcone of Sasaki Ricci solitons in the Sasaki cone, and a unique Sasaki-Einstein metric in each of the two dimensional sub cones.
\end{abstract}


\section{Introduction}
The purpose of this paper is to present a general geometric construction that combines the Sasaki join construction of \cite{BG00a,BGO06} with the Hamiltonian 2-form formalism of \cite{ApCaGa06,ACGT04,ACGT08} to construct many new Sasaki metrics of constant scalar curvature. This is a final version of our ArXiv paper \cite{BoTo14aArX} which is also combined with results from our ArXiv paper \cite{BoTo13bArX}. Partial results from this latter paper will then appear in \cite{BoTo14NY}. 
The method described here has already been used by the authors in special cases \cite{BoTo11,BoTo12b,BoTo13}. This method is the following: consider a regular Sasaki manifold $M$ with its `Boothby-Wang circle bundle' $S^1\ra{1.5}M\ra{1.5} N$ over the K\"ahler manifold $N$. For each pair of relatively prime positive integers $(l_1,l_2)$ we form the Sasaki join $M_{l_1,l_2,\bfw}$ of $M$ with a weighted 3-sphere $S^3_\bfw$ (cf. \cite{BG05}, Example 7.1.12), where the components of the weight vector $\bfw=(w_1,w_2)$ are relatively prime positive integers satisfying $w_1\geq w_2$. The latter has a 2-dimensional Sasaki cone $\gt^+_\bfw$ we call the $\bfw$-Sasaki cone. Now we can deform within $\gt^+_\bfw$ to obtain other Sasakian structures. The quasi-regular ones will fiber over a ruled orbifold $(S_n,\grD)$ with the following structure. $S_n$ is a $\bbc\bbp^1$-bundle over $N$ with an orbifold structure on its fibers giving rise to a branch divisor $\grD$. The orbifold $(S_n,\grD)$ is a projectivization $\bbp(\BOne\oplus L_n)$ where $L_n$ is certain line bundle over $N$, and it admits a Hamiltonian 2-form. The explicit nature of this formalism allows us to obtain extremal (or constant scalar curvature) K\"ahler orbifold metrics on $(S_n,\grD)$. Then by a well known procedure we obtain extremal (constant scalar curvature) Sasaki metrics on the join $M_{l_1,l_2,\bfw}$. This approach was initiated in \cite{BoTo13} for the case $l_2=1$ while the case of arbitrary $l_2$ appears in the thesis \cite{Cas14}. An announcement of our general procedure appears in \cite{BoTo14P}.

Our first main theorem is:

\begin{theorem}\label{admjoincsc}
Let $M_{l_1,l_2,\bfw}=M\star_{l_1,l_2}S^3_\bfw$ be the $S^3_\bfw$-join with a regular Sasaki manifold $M$ which is an $S^1$-bundle over a compact K\"ahler manifold $N$ with constant scalar curvature. Then for each vector $\bfw=(w_1,w_2)\in \bbz^+\times\bbz^+$ with relatively prime components satisfying $w_1>w_2$ there exists a Reeb vector field $\xi_\bfv$ in the 2-dimensional $\bfw$-Sasaki cone on $M_{l_1,l_2,\bfw}$ such that the corresponding ray of Sasakian structures $\cals_a=(a^{-1}\xi_\bfv,a\eta_\bfv,\Phi,g_a)$ has constant scalar curvature.
\end{theorem}

The manifolds $M_{l_1,l_2,\bfw}$ can also be realized as certain three dimensional lens space bundles over $N$.

Suppose in addition that the scalar curvature of $N$ satisfies $s_N\geq 0$, then we obtain more information about extremal Sasaki metrics. In fact, we have
\begin{theorem}\label{allextr}
Suppose that in addition to the hypothesis of Theorem \ref{admjoincsc} the scalar curvature of N satisfies $s_N\geq 0$, then the $\bfw$-Sasaki cone is exhausted by extremal Sasaki metrics. In particular, if the K\"ahler structure on $N$ admits no Hamiltonian vector fields, then the entire Sasaki cone $\grk$ of the join $M_{l_1,l_2,\bfw}=M\star_{l_1,l_2}S^3_\bfw$ can be represented by extremal Sasaki metrics.
\end{theorem}

A particular example of interest when the hypothesis of the last sentence of Theorem \ref{allextr} is satisfied is when $N$ is an algebraic K3 surface. In this case there are many choices of complex structures and many choices of line bundles. But in all cases $M=21\# (S^2\times S^3)$. It is interesting to contemplate the possible diffeomorphism types of the 7-manifolds $21\#(S^2\times S^3)\star_{l_1,l_2} S^3_\bfw$ in this case.

We also give examples where there are more than one CSC ray in the same $\bfw$-Sasaki cone. Indeed, generally we have
\begin{theorem}\label{3cscthm}
Suppose that in addition to the hypothesis of Theorem \ref{admjoincsc} the scalar curvature of N satisfies $s_N> 0$. Then for sufficiently large $l_2$ there are at least three CSC rays in the $\bfw$-Sasaki cone of the join $M_{l_1,l_2,\bfw}$.
\end{theorem}

In particular, Theorem \ref{WZex} below gives a countable infinity of inequivalent contact structures on the two $S^5$-bundles over $S^2$ such that there are at least three CSC rays of Sasaki metrics. However, the bouquet phenomenon which is related to distinct underlying CR structures and appears for $S^3$-bundles over Riemann surfaces \cite{Boy10b,Boy10a,BoTo11,BoTo13} seems not to occur in these more general cases. This is related to the topological rigidity of the Boothby-Wang base space as discussed briefly in Section \ref{diffeosec} below. The non-uniqueness described in Theorem \ref{3cscthm} occurs on a fixed strictly pseudoconvex CR manifold and a fixed contact manifold. The former also  illustrates non-uniqueness on the sub-Riemannian level.

It should be mentioned that generally the CSC rays are most often irregular, that is the closure of a generic Reeb orbit is a torus of dimension greater than one. In this regard in Section \ref{adsas} we fill in a gap that occured in the first version of \cite{BoTo13bArX} concerning the application of the admissibility conditions to irregular Sasakian structures. This was kindly pointed to us by an anonymous referee. It has been shown recently that irregular Sasakian structures have irreducible transverse holonomy \cite{HeSu12b}, and that the corresponding K\"ahler cone is K-semistable \cite{CoSz12} for CSC Sasaki metrics. The latter result had been proven previously in the quasi-regular case in \cite{RoTh11}. Also the non-uniqueness phenomenon of CSC Sasakian structures was first shown to occur for the case of $S^3$-bundles over $S^2$ by a different method in \cite{Leg10}. Theorem \ref{3cscthm} shows that this is fairly common.

Finally, if we assume that the $N$ has a positive K\"ahler-Einstein metric and that the first Chern class of the contact bundle vanishes, we obtain

\begin{theorem}\label{admjoinse}
Let $M_{l_1,l_2,\bfw}=M\star_{l_1,l_2}S^3_\bfw$ be the $S^3_\bfw$-join with a regular Sasaki manifold $M$ which is an $S^1$-bundle over a compact positive K\"ahler-Einstein manifold $N$ with a primitive K\"ahler class $[\gro_N]\in H^2(N,\bbz)$. Assume that the relatively prime positive integers $(l_1,l_2)$ are the relative Fano indices given explicitly by 
$$l_{1}=\frac{\cali_N}{\gcd(w_1+w_2,\cali_N)},\qquad   l_2=\frac{w_1+w_2}{\gcd(w_1+w_2,\cali_N)},$$ where $\cali_N$ denotes the Fano index of $N$.  Then for each vector $\bfw=(w_1,w_2)\in \bbz^+\times\bbz^+$ with relatively prime components satisfying $w_1>w_2$ there exists a Reeb vector field $\xi_\bfv$ in the 2-dimensional $\bfw$-Sasaki cone on $M_{l_1,l_2,\bfw}$ such that the corresponding Sasakian structure $\cals=(\xi_\bfv,\eta_\bfv,\Phi,g)$ is Sasaki-Einstein. Moreover, this ray is the only admissible CSC ray in the $\bfw$-Sasaki cone. 

In addition, for each vector $\bfw=(w_1,w_2)\in \bbz^+\times\bbz^+$ with relatively prime components satisfying $w_1>w_2$ every single ray in the 2-dimensional $\bfw$-Sasaki cone on $M_{l_1,l_2,\bfw}$ admits (up to isotopy) a Sasaki-Ricci soliton.

\end{theorem}

This theorem recovers in a geometric fashion the Sasaki-Einstein metrics obtained previously by physicists in \cite{GMSW04b} by a different method. Some further development of the topology of these manifolds, especially in dimension seven, is given in \cite{BoTo14NY}.

The paper is organized as follows. Section 2 gives a brief review of ruled manifolds that are the projectivization of a complex rank 2 vector bundle of the form $S=\bbp(\BOne\oplus L)$ over a K\"ahler-Einstein manifold $N$. These admit Hamiltonian 2-forms that give rise to the K\"ahler admissible construction that is necessary for our procedure. In the somewhat long Section 3 we describe our join construction, in particular, the join with the weighted 3-sphere, $S^3_\bfw$. We then discuss in detail the orbit structure of quasi-regular Reeb vector fields in the $\bfw$-Sasaki cone. Generally, the quotients appear as orbifold log pairs $(S,\grD)$ which fiber over $N$ with fiber an orbifold of the form $\bbc\bbp^1[v_1,v_2]/\bbz_m$, where $\grD$ is branch divisor, and $\bbc\bbp^1[v_1,v_2]$ is a weighted projective space. In Section 4 we discuss the topology of the joins, giving an algorithm for computing the integral cohomology ring. In Section 5 we present the details of the admissibility conditions on the K\"ahlerian level, while in Section 6 we show how these admissibility conditions lift to the Sasakian level to provide proofs of our main theorems.

\section{Ruled Manifolds}
In this section we consider ruled manifolds of the following form. Let $(N,\gro_N)$ be a compact K\"ahler manifold with primitive integer K\"ahler class $[\gro_N]$, that is, a Hodge manifold. Consider a rank two complex vector bundle of the form $E=\BOne\oplus L$ where $L$ is a complex line bundle on $N$ and $\BOne$ denotes the trivial bundle. By a ruled manifold we shall mean the projectivization $S=\bbp(\BOne\oplus L)$. We can view $S$ as a compactification of the complex line bundle $L$ on $N$ by adding the `section at infinity'. For $x\in N$ we let $(c,z)$ denote a point of the fiber $E_x=\BOne\oplus L_x$. There is a natural action of $\bbc^*$ (hence, $S^1$) on $E$ given by $(c,z)\mapsto (c,\grl z)$ with $\grl\in \bbc^*$. The action $z\mapsto \grl z$ is a complex irreducible representation of $\bbc^*$ determined by the line bundle $L$. Such representations (characters) are labeled by the integers $\bbz$. Thus, we write $L=L_n$ for $n\in \bbz$ and refer to $n$ as the `degree' of $L$.

\subsection{A Construction of Ruled Manifolds}\label{ruledsec}
We now give a construction of such manifolds. Let $S^1\ra{1.6} M\ra{1.6} N$ be the circle bundle over $N$ determined by the class $[\gro_N]\in H^2(N,\bbz)$. We denote the $S^1$-action by $(x,u)\mapsto (x,e^{i\theta}u)$. Now represent $S^3\subset \bbc^2$ as $|z_1|^2+|z_2|^2=1$ and consider an $S^1$-action on $M\times S^3$ given by $(x,u;z_1,z_2)\mapsto (x,e^{i\theta}u;z_1,e^{in\theta}z_2)$. There is also the standard $S^1$-action on $S^3$ given by $(z_1,z_2)\mapsto (e^{i\chi}z_1,e^{i\chi}z_2)$ giving a $T^2$-action on $M\times S^3$ defined by
\begin{equation}\label{T2act}
(x,u;z_1,z_2)\mapsto (x,e^{i\theta}u;e^{i\chi}z_1,e^{i(\chi+n\theta)}z_2).
\end{equation}

\begin{lemma}\label{T2quot}
The quotient by the $T^2$-action of Equation (\ref{T2act}) is the projectivization $S_n=\bbp(\BOne\oplus L_n)$.
\end{lemma}

\begin{proof}
First we see from (\ref{T2act}) that the action is free, so there is a natural bundle projection $(M\times S^3)/T^2\ra{1.6} N$ defined by $\pi(x,[u;z_1,z_2])=x$ where the bracket denotes the $T^2$ equivalence class. The fiber is $\pi^{-1}(x)=[u;z_1,z_2]$ which since $u$ parameterizes a circle is identified with $S^3/S^1=\bbc\bbp^1$. This bundle is trivial if and only if $n=0$ and $n$ labels the irreducible representation of $S^1$ on the line bundle $L_n$.
\end{proof}

We can take the line bundle $L_1$ to be any primitive line bundle in ${\rm Pic}(N)$. In particular, we are interested in the taking $L_1$ to be the line bundle associated to the primitive cohomology class $[\gro_N]\in H^2(N,\bbz)$. Then we have

\begin{lemma}\label{c1L}
The following relation holds: $c_1(L_n)=n[\gro_N]$.
\end{lemma}

\begin{proof}
Equation (\ref{T2act}) implies that the $S^1$-action on the line bundle $L_n$ is given by $z\mapsto e^{in\theta}z$. But we know that the definition of $M$ that it is the unit sphere in the line bundle over $N$ corresponding to $n=1$, and this corresponds to the class $[\gro_N]$, that is $c_1(L_1)=[\gro_N]$. Thus, $c_1(L_n)=n[\gro_N]$.
\end{proof}

\subsection{Ruled Manifolds with known Diffeomorphism Type}\label{diffeosec}
There are several cases when the diffeomorphism type of the ruled manifold can be ascertained. First we have the case when $N=\grS_g$ a Riemann surface of genus $g$. It is well known \cite{McDSa} that in this case there are precisely two diffeomorphism types. They are distinguished by their second Stiefel-Whitney class. This gives rise to inequivalent K\"ahler structures belonging to the same underlying symplectic structure (up to symplectomorphism). It also gives rise to non-conjugate maximal tori in the symplectomorphism group, a fact that was exploited in \cite{Boy10b,Boy10a,BoPa10,BoTo11,BoTo13}. 

On the other hand it appears that this phenomenon changes in higher dimension. It is still known to occur as witnessed by the polygon spaces of \cite{HaTo03} and described in Example 8.5 of \cite{Boy10a}. However, it has been shown recently \cite{ChMaSu10,ChPaSu12} that for $N=\bbc\bbp^p$ with $p>1$ the two ruled manifolds $S_n$ and $S_{n'}$ are diffeomorphic if and only if $|n'|=|n|$. Indeed,the diffeomorphism type is determined completely by its cohomology ring which takes the form
\begin{equation}\label{cohring}
H^*(S_n,\bbz)=\bbz[x_1,x_2]/\bigl(x_1^{p+1},(x_2(nx_1+x_2)\bigr)
\end{equation}
where $x_1,x_2$ have degree $2$. So the Hirzebruch-type phenomenon involving inequivalent complex structures on the same manifold does not generalize when $p>1$.

\subsection{The Admissible Construction}\label{admissible}
We will now assume that $n$ from Section \ref{ruledsec} is non-zero and $(N,\gro_N$)  defines a K\"ahler structure with CSC K\"ahler metric
$g_N$.
Then 
$(\omega_{N_n},g_{N_n}): =(2n\pi \omega_N, 2n\pi g_N)$ satisfies that 
$( g_{N_n}, 
\omega_{N_n})$ or $(- g_{N_n}, 
-\omega_{N_n})$
is a K\"ahler structure (depending on the sign of $n$). 
In either case, we let $(\pm g_{N_n}, \pm \omega_{N_n})$ refer to the K\"ahler structure.
We denote the real dimension of $N$ by
$2 d_{N}$ and write the scalar curvature of $\pm g_{N_n}$ as $\pm 2 d_{N_n} s_{N_n}$.
[So, if e.g. $-g_{N_n}$ is a K\"ahler structure with positive scalar curvature, $s_{N_n}$ would be negative.]

Now Lemma \ref{c1L} implies that
$c_{1}(L_n)= [\omega_{N_n}/2\pi]$.
Then, following \cite{ACGT08}, the total space of the projectivization
$S_n=\bbp(\BOne\oplus L_n)$ is called {\it admissible}.

On these manifolds, a particular type of K\"ahler metric on $S_n$,  also called
{\it admissible}, can now be constructed \cite{ACGT08}. We shall describe this construction in Section \ref{CSCconstruct} where we will use it to prove Theorems \ref{admjoincsc} and \ref{allextr} of the Introduction.
An admissible K\"ahler manifold is a special case of a K\"ahler manifold admitting a so-called Hamiltonian $2$-form \cite{ApCaGa06}. More specifically, the admissible metrics as described in section \ref{CSCconstruct} admit a Hamiltonian $2$-form of order one.

\begin{remark}\label{sNn}
In the special case where $(N,\gro_N$) is K\"ahler-Einstein with K\"ahler metric $g_N$ and Ricci form 
$\rho_N = 2\pi\cali_N \omega_N$, where $\cali_N$ denotes the Fano index, there is a simple relationship between the value of $s_{N_n}$ and the value of $n$. Since the (scale invariant) Ricci form is given by $\rho_N=s_{N_n}\omega_{N_n}$, it is easy to see that $s_{N_n}=\cali_N/n$.
For the general CSC case this will be more complicated and will need to be handled case by case. We do know that if we write
$s_{N_n} = p_n/n$, then $p_n \leq d_N +1$ (see Remark 1 in \cite{ACGT08}).
\end{remark}

\section{The $S^3_\bfw$-Join Construction}\label{thejoin}
The join construction was first introduced in \cite{BG00a} for Sasaki-Einstein manifolds, and later generalized to any quasi-regular Sasakian manifolds in \cite{BGO06} (see also Section 7.6.2 of \cite{BG05}). However, as pointed out in \cite{BoTo11} it is actually a construction involving the orbifold Boothby-Wang construction \cite{BoWa,BG00a}, and so applies to quasi-regular strict contact structures. Although it is quite natural to do so, we do not need to fix the transverse (almost) complex structure. Moreover, in \cite{BoTo13} it was shown that in the special case of $S^3$-bundles over Riemann surfaces a twisted transverse complex structure on a regular Sasakian manifold can be realized by a product transverse complex structure on a certain quasi-regular Sasakian structure in the same Sasaki cone. 

We consider a generalization of the join construction used in previous work \cite{BoTo11,BoTo13}. We refer to \cite{BGO06,BG05} for a thorough discussion of the join construction as well as the fundamentals of Sasakian geometry. Here we let $M$ be a regular Sasakian manifold with constant scalar curvature, and consider the join $M_{l_1,l_2,\bfw}=M\star_{l_1,l_2}S^3_\bfw$ with the weighted 3-sphere $S^3_\bfw$ (a sphere with a weighted circle action) where both $\bfw=(w_1,w_2)$ and $\bfl=(l_1,l_2)$ are pairs of relatively prime positive integers. We can assume that the weights $(w_1,w_2)$ are ordered, namely they satisfy $w_1\geq w_2$. Furthermore, $M_{l_1,l_2,\bfw}$ is a smooth manifold if and only if $\gcd(l_2,l_1w_1w_2)=1$ which is equivalent to $\gcd(l_2,w_i)=1$ for $i=1,2$. Henceforth, we shall assume these conditions. 

The join is constructed from the following  commutative diagram
\begin{equation}\label{s2comdia}
\begin{matrix}  M\times S^3_\bfw &&& \\
                          &\searrow\pi_L && \\
                          \decdnar{\pi_{2}} && M_{l_1,l_2,\bfw} &\\
                          &\swarrow\pi_1 && \\
                         N\times\bbc\bbp^1[\bfw] &&& 
\end{matrix}
\end{equation}
where the $\pi$s are the obvious projections. Here $M$ has a regular contact form $\eta_1$ with Reeb vector field $\xi_1$, and $S^3_\bfw$ has the weighted contact form $\eta_2$ with Reeb vector field $\xi_2=w_1H_1+w_2H_2$ where $H_i$ is the infinitesimal generators of the $S^1$ action on 
$$S^3=\{(z_1,z_2)\in\bbc^2~|~|z_1|^2+|z_2|^2=1\}$$ 
given by sending $z_i$ to $e^{i\theta}z_i$. The circle projection $\pi_L$ is generated by the vector field 
\begin{equation}\label{Lvec}
L_{l_1,l_2,\bfw}=\frac{1}{2l_1}\xi_1-\frac{1}{2l_2}\xi_2.
\end{equation}
Moreover, the 1-form $\eta_{l_1,l_2,\bfw}=l_1\eta_1+l_2\eta_2$ on $M\times S^3$ passes to the quotient $M_{l_1,l_2,\bfw}$ and gives it a contact structure. The Reeb vector field of $\eta_{l_1,l_2,\bfw}$ is the vector field
\begin{equation}\label{Reebjoin}
\xi_{l_1,l_2,\bfw}=\frac{1}{2l_1}\xi_1+\frac{1}{2l_2}\xi_2.
\end{equation}
The base orbifold $N\times \bbc\bbp^1[\bfw]$ has a natural K\"ahler structure, namely the product structure, and this induces a Sasakian structure $\cals_{l_1,l_2,\bfw}=(\xi_{l_1,l_2,\bfw},\eta_{l_1,l_2,\bfw},\Phi,g)$ on $M_{l_1,l_2,\bfw}$. The transverse complex structure $J=\Phi |_{\cald_{l_1,l_2,\bfw}}$ is the lift of  the product complex structure on $N\times \bbc\bbp^1[\bfw]$.

It follows from Proposition 7.6.7 of \cite{BG05} that the join $M_{l_1,l_2,\bfw}$ can also be realized as a fiber bundle over $N$ associated to the principal $S^1$-bundle $M\ra{1.6} N$ with fiber the lens space $L(l_2;l_1w_1,l_1w_2)$.
It is easy to see that the join of extremal (CSC) Sasaki metrics gives an extremal (CSC) Sasaki metric induced by the product extremal (CSC) K\"ahler metrics. Thus, since weighted projective spaces have extremal orbifold metrics, we can take the Sasakian structure $\cals_{l_1,l_2,\bfw}$ to be extremal. However, most of the CSC Sasaki metrics of interest in this work are not induced by the product of CSC K\"ahler metrics.

\subsection{The First Chern Class}
Let us compute the first Chern class of our induced contact structure $\cald_{l_1,l_2,\bfw}$ on $M\star_{l_1,l_2}S^{3}_\bfw$. The orbifold first Chern class of the base is 
\begin{equation}\label{c1N}
c_1^{orb}(N\times \bbc\bbp^1[\bfw])=c_1(N)+\frac{|\bfw|}{w_1w_2}PD(D)
\end{equation}
as an element of $H^2(N\times \bbc\bbp^1[\bfw],\bbq)\approx H^2(N,\bbq)\oplus H^2(\bbc\bbp^1[\bfw]),\bbq)$ where $D$ a divisor given by $z_1=0$ or $z_2=0$ and $PD$ denotes Poincar\'e dual. The K\"ahler form on $N\times \bbc\bbp^1[\bfw]$ is $\gro_{l_1,l_2}=l_1\gro_N+ l_2\gro_\bfw$ where $\gro_\bfw$ is the standard K\"ahler form on $\bbc\bbp^1[\bfw]$ which satisfies $[\gro_\bfw]=\frac{[\gro_0]}{w_1w_2}$ where $\gro_0$ is the standard volume form on $\bbc\bbp^1$. Note that $PD(D)=[\gro_0]$. Pulling $\gro_{l_1,l_2}$ back to the join $M_{l_1,l_2,\bfw}= M\star_{l_1,l_2}S^{3}_\bfw$ we have $\pi^*\gro_{l_1,l_2}=d\eta_{l_1,l_2,\bfw}$ implying that $l_1\pi^*[\gro_N]+l_2\pi^*[\gro_\bfw]=0$ in $H^2(M_{l_1,l_2,\bfw},\bbz)$. So taking $\pi^*[\gro_N]= l_2\grg$ and $\pi^*[\gro_\bfw]=-l_1\grg$ for some generator $\grg\in H^2(M_{l_1,l_2,\bfw},\bbz)$ we have
\begin{equation}\label{c1cald}
c_1(\cald_{l_1,l_2,\bfw})=\pi^*c_1(N)-l_1|\bfw|\grg.
\end{equation}
Taking the mod 2 reduction gives the second Stiefel-Whitney class of $M_{l_1,l_2,\bfw}$, viz.
\begin{equation}\label{w2M}
w_2(M_{l_1,l_2,\bfw})=\pi^*w_2(N)-\grr(l_1|\bfw|\grg)
\end{equation}
where $\grr$ is the reduction mod 2 map. This implies

\begin{corollary}\label{spincor}
If $l_1$ is even or if $w_i$ are both odd for $i=1,2$, then $M_{l_1,l_2,\bfw}$ is a spin manifold if and only if $N$ is a spin manifold. On the other hand if both $l_1$ and $|\bfw|$ are odd, then $M_{l_1,l_2,\bfw}$ is a spin manifold if and only if $N$ is not a spin manifold.
\end{corollary}

Equation \eqref{c1cald} reduces further in the special case that $[\gro_N]$ is monotone. Actually we are interested in a generalization. We say  that $[\gro_N]$ is {\it quasi-monotone} if  $c_1(N)=\cali_N[\gro_N]$ for some integer $\cali_N$. Here $\cali_N$ is the {\it Fano index} when $\cali_N$ is positive (the monotone case) and the {\it canonical index} when it is negative. We also allow the case $\cali_N=0$. So when $[\gro_N]$ is quasi-monotone we have
\begin{equation}\label{c1cald2}
c_1(\cald_{l_1,l_2,\bfw})=(l_2\cali_N -l_1|\bfw|)\grg.  
\end{equation}

Of particular interest is the cohomological Einstein condition. Let $c_1(\calf_\xi)$ be the basic first Chern class of the Sasakian structure $\cals_{l_1,l_2,\bfw}$, then the cohomological Einstein condition is $c_1(\calf_\xi)=a[d\eta_{l_1,l_2,\bfw}]_B$ for some positive constant $a$, where $[\cdot]_B$ denotes the basic cohomology class in $H^{1,1}(\calf_\xi)$. This implies that $c_1(\cald_{l_1,l_2,\bfw})$ is a torsion class, but for convenience we shall assume that $c_1(\cald_{l_1,l_2,\bfw})=0$ which implies the condition $l_2\cali_N=|\bfw|l_{1}$. We have arrived at:

\begin{lemma}\label{c10}
Necessary conditions for the Sasaki manifold $M_{l_1,l_2,\bfw}$ to admit a Sasaki-Einstein metric is that $\cali_N>0$, and that  
$$l_2(\bfw)=\frac{|\bfw|}{\gcd(|\bfw|,\cali_N)},\qquad l_{1}(\bfw)=\frac{\cali_N}{\gcd(|\bfw|,\cali_N)}.$$
\end{lemma}
The integers $l_1(\bfw),l_2(\bfw)$ in Lemma \ref{c10} are called {\it relative Fano indices} \cite{BG00a}.

\subsection{The Sasaki Cone}\label{sasconesec}
Since for any Sasakian structure $\cals$ the Reeb vector field lies in the center of the Lie algebra $\ga\gu\gt(\cals)$ of  the Sasaki automorphism group $\gA\gu\gt(\cals)$, it follows from the join construction that the Lie algebra $\ga\gu\gt(\cals_{l_1,l_2,\bfw})$ of the automorphism group of the join satisfies $\ga\gu\gt(\cals_{l_1,l_2,\bfw})=\ga\gu\gt(\cals_1)\oplus\ga\gu\gt(\cals_{\bfw})\mod (L_{l_1,l_2,\bfw})$ where $\cals_1$ is the Sasakian structure on $M$, and $\cals_\bfw$ is the Sasakian structure on $S^3_\bfw$. Now the unreduced Sasaki cone \cite{BGS06} $\gt^+$ of $\cals=(\xi,\eta,\Phi,g)$ is by definition the positive cone in the Lie algebra $\gt$ of a maximal torus in $\gA\gu\gt(\cals)$, i.e.
\begin{equation}\label{sascone}
\gt^+=\{X\in\gt~|~\eta(X)>0\}.
\end{equation}
Thus, the Sasaki cone $\gt^+_{l_1,l_2,\bfw}$ of the join $M_{l_1,l_2,\bfw}$ satisfies
\begin{equation}\label{sasconejoin}
\gt^+_{l_1,l_2,\bfw}=\{X\in \gt_{l_1,l_2,\bfw} ~|~\eta_{l_1,l_2,\bfw}(X)>0\}=\gt^+_1+\gt^+_\bfw \mod (L_{l_1,l_2,\bfw}).
\end{equation}
If the Lie algebra of a maximal torus of the automorphism group of $\cals_1$ has dimension $k$, then $\dim \gt^+_{l_1,l_2,\bfw}=k+2$, since the $\gt_\bfw$ has dimension $2$. However, in this paper we are mainly concerned with the 2-dimensional subcone $\gt^+_\bfw$, which we call the $\bfw$-Sasaki cone, of the full Sasaki cone $\gt^+_{l_1,l_2,\bfw}$. The $\bfw$-Sasaki cone $\gt_\bfw^+$ can be identified with the first quadrant in $\bbr^2$ with coordinates $v_1,v_2$ for all $\bfw$, viz.
\begin{equation}\label{tw+}
\gt_\bfw^+ =\{(v_1,v_2)\in\bbr^2~|~v_1,v_2>0\}.
\end{equation}

We are also interested in the full reduced Sasaki cone $\grk$ which is $\gt^+/\calw$ where $\calw$ is the Weyl group of the Sasaki automorphism group $\gA\gu\gt(\cals)$. One can think of $\grk$ as the moduli space of Sasakian structures with a fixed underlying CR structure $(\cald,J)$.

\subsection{The Tori Actions}
Consider the action of the $3$-dimensional torus $T^{3}$ on the product $M\times S^{3}_\bfw$ defined by
\begin{equation}\label{Taction}
(x,u;z_1,z_2)\mapsto (x,e^{il_2\theta}u;e^{i(\phi_1-l_1w_1\theta)}z_1,e^{i(\phi_2-l_1w_2\theta)}z_2).
\end{equation}
The Lie algebra $\gt_{3}$ of $T^{3}$ is generated by the vector fields $L_{l_1,l_2,\bfw},H_1,H_2$. Then the join $M_{l_1,l_2,\bfw}=M\star_{l_1,l_2}S^3_\bfw$ defined in the beginning of Section \ref{thejoin} is the quotient of $M\times S^{3}_\bfw$ by the $S^1$ subgroup of $T^3$ defined by setting $\phi_1=\phi_2=0$. Alternatively it is the fiber bundle 
$$M_{l_1,l_2,\bfw}=M\times_{S^1}L(l_2;l_1w_1,l_1w_2)$$ 
over the K\"ahler manifold $N$ associated to the principal $S^1$-bundle $M\ra{1.5} N$ with fiber the lens space $L(l_2;l_1w_1,l_1w_2)$. The $S^1$ action on the lens space is accomplished in two stages. First, represent $L(l_2;l_1w_1,l_1w_2)$ as a $\bbz_{l_2}$ quotient of $S^3_\bfw$, then the residual $S^1_\theta/\bbz_{l_2}\approx S^1$ action is 
\begin{equation}\label{resS1act}
(x,u;z_1,z_2)\mapsto (x,e^{i\theta}u;[e^{-i\frac{l_1w_1}{l_2}\theta)}z_1,e^{-i\frac{l_1w_2}{l_2}\theta)}z_2]).
\end{equation}
The brackets in Equation (\ref{resS1act}) denote the equivalence class defined by $(z'_1,z'_2)\sim (z_1,z_2)$ if $(z'_1,z'_2)=(\grl^{l_1w_1}z_1,\grl^{l_1w_2}z_2)$ for $\grl^{l_2}=1$.

Next consider the $T^2$ action of $S^1_\phi\times (S^1_\theta/\bbz_{l_2})$ on $M\times L(l_2;l_1w_1,l_1w_2)$ given by
\begin{equation}\label{T2action3}
(x,u;z_1,z_2)\mapsto (x,e^{i\theta}u;[e^{i(v_1\phi-\frac{l_1w_1}{l_2}\theta)}z_1,e^{i(v_2\phi-\frac{l_1w_2}{l_2}\theta)}z_2]),
\end{equation}
This gives rise to the commutative diagram 
\begin{equation}\label{comdia2}
\begin{matrix}  M\times L(l_2;l_1w_1,l_1w_2) &&& \\
                          &\searrow\pi_L && \\
                          \decdnar{\pi_B} && M_{l_1,l_2,\bfw} &\\
                          &\swarrow\pi_\bfv && \\
                          B_{l_1,l_2,\bfv,\bfw} &&& 
\end{matrix}
\end{equation}
where $B_{l_1,l_2,\bfv,\bfw}$ is a bundle over $N$ with fiber a weighted projective space, and $\pi_B$ denotes the quotient projection by $T^2$. The Lie algebra of this $T^2$ is generated by 
\begin{equation}\label{infq+2action}
L_\bfw=\frac{1}{2l_1}\xi_M-\sum_{j=0}^{q}\frac{1}{2l_2}w_jH_j, \qquad \xi_\bfv=\sum_jv_jH_j,
\end{equation}
where $\xi_M$ denotes the Reeb vector field of the regular Sasakian structure on $M$.
Note that $\xi_\bfv$ is a Reeb vector field in the $\bfw$-Sasaki cone of $M_{l_1,l_2,\bfw}$.

Let us analyze the behavior of the $T^2$ action given by Equation (\ref{T2action3}). We shall see that it it not generally effective. First we notice that the $S^1_\theta$ action is free since it is free on the first factor. Next we look for fixed points under a subgroup of the circle $S^1_\phi$. Thus, we impose 
$$(e^{iv_1\phi}z_1,e^{iv_2\phi}z_2)=(e^{-2\pi\frac{l_1w_1}{l_2}ri}z_1,e^{-2\pi\frac{l_1w_2}{l_2}ri}z_2)$$
for some $r=0,\ldots,l_2-1$. If $z_1z_2\neq 0$ we must have
\begin{equation}\label{phisoln}
v_1\phi=2\pi(-\frac{l_1w_1r}{l_2} +k_1),\qquad v_2\phi=2\pi(-\frac{l_1w_2r}{l_2} +k_2)
\end{equation}
for some integers $k_1,k_2$ which in turn implies
$$l_1r(w_2v_1-w_1v_2)=l_2(k_2v_1-k_1v_2).$$
This gives
\begin{equation}\label{reqn}
r=\frac{l_2}{l_1}\frac{k_2v_1-k_1v_2}{w_2v_1-w_1v_2}
\end{equation}
which must be a nonnegative integer less than $l_2$. 
We can also solve Equations (\ref{phisoln}) for $\phi$ by eliminating $\frac{l_1r}{l_2}$ giving
\begin{equation}\label{phisoln2}
\phi =2\pi \frac{k_1w_2-k_2w_1}{w_2v_1-w_1v_2}.
\end{equation}

Next we write (\ref{reqn}) as
\begin{equation}\label{reqn3}
r=\Bigl(\frac{l_2}{\gcd(|w_2v_1-w_1v_2|,l_2)}\Bigr)\Bigl(\frac{k_2v_1-k_1v_2}{l_1\frac{w_2v_1-w_1v_2}{\gcd(|w_2v_1-w_1v_2|,l_2)}}\Bigr)
\end{equation}
Since $v_1$ and $v_2$ are relatively prime, we can choose $k_1$ and $k_2$ so that the term in the last parentheses is $1$. This determines $r$ as 
\begin{equation}\label{reqn4}
r=\frac{l_2}{\gcd(|w_2v_1-w_1v_2|,l_2)}
\end{equation}

Now suppose that $z_2=0$. Then generally we have $e^{iv_1\phi}=e^{-2\pi\frac{l_1w_1}{l_2}ri}$ for some $r=0,\ldots,l_2-1$ or equivalently $r=1,\ldots,l_2$. This gives 
\begin{equation}\label{z20}
\phi=2\pi(-\frac{l_1w_1r}{v_1l_2}+\frac{k}{v_1}).
\end{equation} 
A similar computation at $z_1=0$ gives 
\begin{equation}\label{z10}
\phi=2\pi(-\frac{l_1w_2r'}{v_2l_2}+\frac{k'}{v_2}).
\end{equation}
We are interested in when regularity can occur. For this we need the minimal angle at the two endpoints to be equal. This gives
$$-\frac{l_1w_2r'}{v_2l_2}+\frac{k'}{v_2}=-\frac{l_1w_1r}{v_1l_2}+\frac{k}{v_1}$$
for some choice of integers $k,k'$ and nonnegative integers $r,r'<l_2$. This gives
\begin{equation}\label{endpteqn}
\frac{-l_1w_2r'+k'l_2}{v_2}=\frac{-l_1w_1r+kl_2}{v_1}.
\end{equation}

\subsection{Periods of Reeb Orbits}
We assume that $\bfw\neq (1,1)$. We want to determine the periods of the orbits of the flow of the Reeb vector field defined by the weight vector $\bfv=(v_1,v_2)$. In particular, we want to know when there is a regular Reeb vector field in the $\bfw$-Sasaki cone.

Let us now generally determine the minimal angle, hence the generic period of the Reeb orbits, on the dense open subset $Z$ defined by $z_1z_2\neq 0$. For convenience we set $s=\gcd(|w_2v_1-w_1v_2|,l_2)$ in which case \eqref{reqn4} becomes $r=l_2/s$.

\begin{lemma}\label{generic3}
The minimal angle on $Z$ is $\frac{2\pi}{s}$.  Thus,  $S^1_\phi/\bbz_s$ acts freely on the dense open subset $Z$. 
\end{lemma}

\begin{proof}
We choose $k_1,k_2$ in Equation (\ref{reqn3}) so that the last parentheses equals $1$. This gives
$$l_1\frac{w_2v_1-w_1v_2}{s}=k_2v_1-k_1v_2.$$
Rearranging this becomes
$$(sk_2-l_1w_2)v_1=(sk_1-l_1w_1)v_2.$$
Since $v_1$ and $v_2$ are relatively prime this equation implies $sk_i=l_1w_i+mv_i$ for $i=1,2$ and some integer $m$. Putting this into Equation (\ref{phisoln2}) gives $\phi=\frac{2\pi m}{s}$, so the minimal angle is $\frac{2\pi}{s}$.
\end{proof}

We now investigate the endpoints defined by $z_2=0$ and $z_1=0$.

\begin{proposition}\label{regularprop}
The following hold:
\begin{enumerate}
\item The period on $Z$, namely $\frac{2\pi}{s}$, is an integral multiple of the periods at the endpoints. Hence, $S^1_\phi/\bbz_s$ acts effectively on $M_{l_1,l_2,\bfw}$.
\item The period at the endpoint $z_j=0$ is $\frac{2\pi}{v_il_2}$ where $i\equiv j+1\mod 2$. So the end points have equal periods if and only if $\bfv=(1,1)$. 
\item The $\bfw$-Sasaki cone contains a regular Reeb vector field if and only if  $l_2$ divides $w_1-w_2$, and in this case it is given by $\bfv=(1,1)$.
\end{enumerate}
\end{proposition}

\begin{proof}
A Reeb vector field will be regular if and only if the period of its orbit is the same at all points. We know that it is $\frac{2\pi}{s}$ on $Z$. We need to determine the minimal angle at the endpoints. From Equation (\ref{z20}) the angle at $z_2=0$ is
$$\phi=2\pi(\frac{-l_1w_1r+kl_2}{v_1l_2}). $$
Now $\gcd(l_2,l_1w_1)=1,$
so we can choose $k$ and $r$ such that numerator of the term in the large parentheses is $1$. This gives period $\frac{2\pi}{v_1l_2}$. Similarly, at $z_1=0$ we have the period $\frac{2\pi}{v_2l_2}$. So the period is the same at the endpoints if and only if $v_1=v_2$ which is equivalent to $\bfv=(1,1)$ since $v_1$ and $v_2$  are relatively prime which proves $(2)$. 

Moreover, the period is the same at all points if and only if
\begin{equation}\label{eqper}
\bfv=(1,1),\qquad l_2=s=\gcd(|w_2v_1-w_1v_2|,l_2).
\end{equation} 
But the last equation holds if and only if $l_2$ divides $w_1-w_2$ proving $(3)$. 

(1) follows from the fact that for each $i=1,2$, $v_il_2$ is an integral multiple of $\gcd(|w_2v_1-w_1v_2|,l_2)=s$.
\end{proof}

In contrast to the 2-dimensional Sasaki cones in \cite{BoTo13}, not every $\bfw$-Sasaki cone has a regular Reeb vector field. Nevertheless, it does have a special Reeb vector field, namely that given by $\bfv=(1,1)$. For this there can be, as usual, two branch divisors, but they have the same ramification index, namely $m=l_2/s$. We refer to this Reeb field as {\it almost regular}. Clearly, there is precisely one almost regular Reeb vector field in each $\bfw$-Sasaki cone of $M_{l_1,l_2,\bfw}$.

\begin{example}\label{l_2=1} Regular Reeb vector fields. As stated in (c) of Proposition \ref{regularprop} when $l_2$ divides $w_1-w_2$ we always have a regular Reeb vector field in the $\bfw$-Sasaki cone by taking $\bfv=(1,1)$. (This was the case in \cite{BoTo13} where $l_2=1$.) We obtain $M_{l_1,l_2,\bfw}$ as a principle $S^1$ bundle over the smooth quotient $B_{l_1,l_2,1,\bfw}=S_n=\bbp(\BOne\oplus L_n)$ with $n=l_1\frac{w_1-w_2}{l_2}$. 
\end{example}

\subsection{$B_{l_1,l_2,\bfv,\bfw}$ as a Log Pair}
We follow the analysis in Section 3 of \cite{BoTo13}. We have the action of the 2-torus $S^1_\phi/\bbz_s\times (S^1_\theta/\bbz_{l_2})$ on $M\times L(l_2;l_1w_1,l_1w_2)$ given by Equation (\ref{T2action3}), and denoted by $\cala_{\bfv,l,\bfw}$, whose quotient space is $B_{l_1,l_2,\bfv,\bfw}$. It follows from Equation (\ref{T2action3}) that $B_{l_1,l_2,\bfv,\bfw}$ is a bundle over $N$ with fiber a weighted projective space of complex dimension one. By (1) of Proposition \ref{regularprop} the generic period is an integral multiple, say $m_i$, of the period at the divisor $D_i$. Thus, for $i=1,2$ we have 
\begin{equation}\label{ramind}
m_i=v_i\frac{l_2}{s}=v_im.
\end{equation}
Note that from its definition $m=\frac{l_2}{s}$, so $m_i$ is indeed a positive integer. It is the ramification index of the branch divisor $D_i$. We think of $D_1$ as the zero section and $D_2$ as the infinity section of the bundle $B_{l_1,l_2,\bfv,\bfw}$. Thus, $B_{l_1,l_2,\bfv,\bfw}$ is a fiber bundle over $N$ with fiber $\bbc\bbp^1[v_1,v_2]/\bbz_m\approx \bbc\bbp^1$. The isomorphism is simply $[z_1,z_2]\mapsto [z_1^{m_2},z_2^{m_1}]$ where the brackets denote the obvious equivalence classes on $\bbc\bbp^1[v_1,v_2]/\bbz_m$. The complex structure of $B_{l_1,l_2,\bfv,\bfw}$ is the projection of the transverse complex structure on $M_{l_1,l_2,\bfw}$ which in turn is the lift of the product complex structure on $N\times \bbc\bbp^1[\bfw]$. However, $B_{l_1,l_2,\bfv,\bfw}$ is not generally a product as a complex orbifold, nor even topologically.

Now we can follow the analysis leading to Lemma 3.14 of \cite{BoTo13}. So we define the map 
$$\th_\bfv:M\times L(l_2;l_1w_1,l_1w_2)\ra{1.6} M\times L(l_2;l_1w_1v_2,l_1w_2v_1)$$ 
by
\begin{equation}\label{th}
\th_\bfv(x,u;[z_1,z_2])=(x,u;[z_1^{m_2},z_2^{m_1}]).
\end{equation}
It is a $mv_1v_2$-fold covering map. Similar to \cite{BoTo13} we get a commutative diagram:
\begin{equation}\label{actcomdia}
\begin{matrix}
M\times L(l_2;l_1w_1,l_1w_2) &\fract{\cala_{\bfv,l,\bfw}(\grl,\grt)}{\ra{2.5}} & M\times  L(l_2;l_1w_1,l_1w_2) \\
\decdnar{\th_\bfv} && \decdnar{\th_\bfv} \\
M\times  L(l_2;l_1w'_1,l_1w'_2) & \fract{\cala_{(1,1),l,\bfw'}(\grl,\grt^{mv_1v_2})}{\ra{2.5}} & M\times  L(l_2;l_1w'_1,l_1w'_2),
\end{matrix}
\end{equation}
where $\bfw'=(v_2w_1,v_1w_2)$ and $\grt=e^{i\phi},\grl=e^{i\theta}$. So $B_{l_1,l_2,\bfv,\bfw}$ is the log pair $(B_{l_1,l_2,1,\bfw'},\grD)$ with 
branch divisor 
\begin{equation}\label{branchdiv}
\grD=(1-\frac{1}{m_1})D_1+ (1-\frac{1}{m_2})D_2,
\end{equation}
where $B_{l_1,l_2,1,\bfw'}$ is a $\bbc\bbp^1$-bundle over $N$. Now $B_{l_1,l_2,\bfv,\bfw}$ is the quotient $\bigl(M\times L(l_2;l_1w_1,l_1w_2)\bigr)/\cala_{\bfv,l,\bfw}(\grl,\grt)$, and $B_{l_1,l_2,1,\bfw'}$ is the quotient $\bigl(M\times  L(l_2;l_1w'_1,l_1w'_2)\bigr)/\cala_{(1,1),l,\bfw'}(\grl,\grt^{mv_1v_2})$. So $\th_\bfv$ induces a map $h_\bfv:B_{l_1,l_2,\bfv,\bfw}\ra{1.6}B_{l_1,l_2,1,\bfw'}$ defined by 
\begin{equation}\label{hquot}
h_\bfv([x,u;[z_1,z_2]])=[x,u;[z_1^{m_2},z_2^{m_1}]],
\end{equation}
where the outer brackets denote the equivalence class with respect to the corresponding $T^2$ action. We have

\begin{lemma}\label{biholo}
The map $h_\bfv:B_{l_1,l_2,\bfv,\bfw}\ra{1.6}B_{l_1,l_2,1,\bfw'}$ defined by Equation (\ref{hquot}) is a biholomorphism.
\end{lemma}

\begin{proof}
The map is ostensibly holomorphic. Now $\th_\bfv$ is the identity map on $M$ and a $mv_1v_2$-fold covering map on the corresponding lens spaces. From the commutative diagram (\ref{actcomdia}) the induced map $h_\bfv$ is fiber preserving and is a bijection on the fibers with holomorphic inverse.
\end{proof}

\begin{remark}\label{Galcov}
It is well known that a weighted projective line $\bbc\bbp^1[w_1,w_2]$ is biholomorphic to the projective line itself $\bbc\bbp^1$. Similarly, developable orbifolds of the form $\bbc\bbp^1/G$ are biholomorphic to $\bbc\bbp^1$ for any finite reflection group $G\subset \gA\gu\gt(\bbc\bbp^1)$. In the case of our ruled manifolds this gives rise to Galois covers of log pairs
$$(S_n,(1-\frac{1}{m_1})D_1+(1-\frac{1}{m_2}D_2))\ra{1.8}(S_n,(1-\frac{1}{m})(D_1+D_2))\ra{1.8} (S_n,\emptyset).$$
Set theoretically the maps are the identity maps with the identity Galois group. However, they are inequivalent as orbifolds. For further discussion of this approach see \cite{GhKo05}. Note also that generally the trivial orbifold $(S_n,\emptyset)$ does not occur as one of our quotients. 
\end{remark}

Lemma \ref{biholo} allows us to consider the orbifold $B_{l_1,l_2,\bfv,\bfw}$ as the log pair $(B_{l_1,l_2,1,\bfw'},\grD)$ where $\grD$ is given by Equation \eqref{branchdiv}. Notice, as mentioned above, when $\bfv=(1,1)$ we have an almost regular Reeb vector field. Here the orbifold structure can be non-trivial, namely, $B_{l_1,l_2,(1,1),\bfw}=(B_{l_1,l_2,1,\bfw'},\grD)$ where $m_1=m_2=m=\frac{l_2}{s}$ and the branch divisor is given by
$$\grD=(1-\frac{1}{m})(D_1+D_2).$$

The $T^2$ action $\cala_{(1,1),l,\bfw'}:M\times  L(l_2;l_1w'_1,l_1w'_2)\ra{1.5} M\times  L(l_2;l_1w'_1,l_1w'_2)$ is given by
\begin{equation}\label{T2action4}
(x,u;z_1,z_2)\mapsto (x,e^{i\theta}u;[e^{i(\phi-\frac{l_1w'_1}{l_2}\theta)}z_1,e^{i(\phi-\frac{l_1w'_2}{l_2}\theta)}z_2]),
\end{equation}
Defining $\chi=\phi-\frac{l_1w'_1}{l_2}\theta$ gives
\begin{equation}\label{T2action5}
(x,u;z_1,z_2)\mapsto (x,e^{i\theta}u;[e^{i\chi}z_1,e^{i(\chi+\frac{l_1}{l_2}(w'_1-w'_2)\theta)}z_2]).
\end{equation}
The analysis above shows that this action is generally not free, but has branch divisors at the zero ($z_2=0$) and infinity ($z_1=0$) sections with ramification indices both equal to $m$. 

Equation (\ref{T2action5}) tells us that the $T^2$-quotient space $B_{l_1,l_2,1,\bfw'}$ is the projectivization of the holomorphic rank two vector bundle $E=\BOne \oplus L_n$ over $N$ where $\BOne$ denotes the trivial line bundle and $L_n$ is a line bundle of `degree' $n=\frac{l_1}{s}(w_1v_2-w_2v_1)$ with $s=\gcd(|w_1v_2-w_2v_1|,l_2)$. So $S_n=\bbp(\BOne\oplus L_n)$ is a smooth projective algebraic variety. Next we identify $N$ with the zero section $D_1$ of $L_n$, and note that $c_1(L_n)$ is just the restriction of the Poincar\'e dual  of $D_1$ to $D_1$, i.e. $PD(D_1)|_{D_1}=c_1(L_n)$.

Summarizing we have

\begin{theorem}\label{preSE}
Let $M_{l_1,l_2,\bfw}=M\star_{l_1,l_2}S^3_\bfw$ be the join as described in the beginning of the section with the induced contact structure $\cald_{l_1,l_2,\bfw}$. Let $\bfv=(v_1,v_2)$ be a weight vector with relatively prime integer components and let $\xi_\bfv$ be the corresponding Reeb vector field in the Sasaki cone $\gt^+_\bfw$. Then the quotient of $M_{l_1,l_2,\bfw}$ by the flow of the Reeb vector field $\xi_\bfv$ is a projective algebraic orbifold written as a the log pair $(S_n,\grD)$ where $S_n$ is the total space of the projective bundle $\bbp(\BOne\oplus L_n)$ over the K\"ahler manifold $N$ with $n=l_1\bigl(\frac{w_1v_2-w_2v_1}{s}\bigr)$, $\grD$ the branch divisor
\begin{equation}\label{branchdiv2}
\grD=(1-\frac{1}{m_1})D_1+ (1-\frac{1}{m_2})D_2,
\end{equation}
with ramification indices $m_i=v_i\frac{l_2}{s}=v_im$ and divisors $D_1$ and $D_2$ given by the zero section $\BOne\oplus 0$ and infinity section $0\oplus L_n$, respectively. The fiber of the orbifold $(S_n,\grD)$ is the orbifold $\bbc\bbp[v_1,v_2]/\bbz_m$.
\end{theorem}


Next we focus on the projective bundle $S_{n}=\bbp(\BOne\oplus L_n)$. From Equation (\ref{T2action5}) we see that $S_{n}$ is a fiber bundle over $N$ with fiber $\bbc\bbp^1$ associated to the principle $S^1$-bundle $M\ra{1.6} N$. We want to determine the K\"ahler class $[\gro_B]$ of the orbifold $B_{l_1,l_2,\bfv,\bfw}=(S_n,\grD)$ induced by the projection $M_{l_1,l_2,\bfw}\ra{1.6} B_{l_1,l_2,\bfv,\bfw}$. First consider the following commutative diagram:

\begin{equation}\label{t3comdia}
\begin{matrix} &&  M\times L(l_2;l_1w_1,l_1w_2) && \\
                          &&\decdnar{\pi_L} && \\
                        && M_{l_1,l_2,\bfw}&&\\
                          &\swarrow\pi_\bfw &&\searrow\pi_\bfv & \\
                         N\times\bbc\bbp^1[\bfw] &&&& (S_n,\grD) \\
                         &p_\bfw \searrow &&\swarrow p_\bfv & \\
                         && N &&
\end{matrix}
\end{equation}
where $p_\bfw,p_\bfv$ are the obvious projections. 
Second, note that we have the following lemma

\begin{lemma}\label{BoTo13b}
For the log pair $(S_n,\Delta)$ with
$$\grD= (1-1/m_1)D_1+(1-1/m_2)D_2$$ 
the orbifold Chern class equals
$$c_1^{orb}(S_n,\Delta) = p_\bfv^*c_1(N) + \frac{1}{m_1}PD(D_1)+\frac{1}{m_2}PD(D_2).$$
\end{lemma} 

\begin{proof}
The usual argument gives that
$$c_1^{orb}(S_n,\Delta) = p_\bfv^*c_1(S_n) +  (\frac{1}{m_1}-1)PD(D_1)+(\frac{1}{m_2}-1)PD(D_2)$$
and the lemma now follows from the fact that
$$c_1(S_n) = p_\bfv^*c_1(N) + PD(D_1)+PD(D_2).$$
One can verify the last fact by using the explicit Ricci form above for some convenient choice of admissible metric (e.g. take $F(\gz)=(1-\gz^2)\gp(\gz)$) in the case $m_1=m_2=1$, but it should also follow from general principles.
\end{proof}

By Equation \eqref{c1cald}
$$\pi_\bfv^*c_1^{orb}(S_n,\grD)=c_1(\cald_{l_1,l_2,\bfw})=(p_\bfw\circ\pi_\bfw)^*c_1(N)-l_1|\bfw|\grg.$$ 
So from Lemma \ref{BoTo13b} we have 
$$(p_\bfv\circ \pi_\bfv)^*c_1(N)+ \frac{1}{m_1}\pi_\bfv^*PD(D_1)+\frac{1}{m_2}\pi_\bfv^*PD(D_2) =(p_\bfw\circ \pi_\bfw)^*c_1(N)-l_1|\bfw|\grg.$$
We also know that (see e.g. Section 1.3 in \cite{ACGT08}) 
$$PD(D_1)-PD(D_2)=np_\bfv^*[\gro_N]$$
and so
$$\pi_\bfv^*PD(D_1)-\pi_\bfv^*PD(D_2 ) = n (p_\bfv\circ \pi_\bfv)^*[\omega_N].$$
From the commutative diagram \eqref{t3comdia} we see that 
$$(p_\bfv\circ \pi_\bfv)^*[\omega_N] = (p_\bfw\circ \pi_\bfw)^*[\omega_N] =  l_2 \gamma.$$
and 
$$(p_\bfv\circ \pi_\bfv)^*c_1(N)=(p_\bfw\circ \pi_\bfw)^*c_1(N),$$
so we get the system
$$
\begin{array}{rcl}
\frac{1}{m_1}\pi_\bfv^*PD(D_1)+\frac{1}{m_2}\pi_\bfv^*PD(D_2) &=& -l_1|\bfw|\grg.\\
\\
\pi_\bfv^*PD(D_1)-\pi_\bfv^*PD(D_2 )  & = &  l_2n \gamma 
\end{array}
$$
which implies that $\pi_\bfv^*PD(D_1) =  \frac{\frac{nl_2}{m_2} - l_1 |\bfw|}{\frac{1}{m_1} + \frac{1}{m_2}}= -m_1l_1w_2 \gamma$ and $\pi_\bfv^*PD(D_2) = -m_2 l_1w_1 \gamma$.

We are now ready to prove the following lemma

\begin{lemma}\label{kahclassB}
The induced K\"ahler class on $B_{l_1,l_2,\bfv,\bfw}=(S_n,\grD)$ takes the form
$$k_1p_\bfv^*[\gro_N]+k_2PD(D_1)$$
for some positive integers $k_1,k_2$.
\end{lemma}

\begin{proof}
From the commutative diagram \eqref{t3comdia} we see that on degree $2$ homology $\ker\pi_B^*=(\pi_\bfv\circ\pi_L)^*$ has dimension $2$. We claim that $p_\bfv^*[\gro_N]$ and $PD(D_1)$ span $\ker\pi_B^*$. To see this we note that from the definition of the join, that $p_\bfv^*[\gro_N]$ is in $\ker\pi_B^*$. Moreover, 
$$(p_\bfv\circ\pi_\bfv\circ\pi_L)^*:H^2(N,\bbr)\ra{1.6} H^2(M\times L(l_2;l_1w_1,l_1w_2),\bbr)$$ 
has a one dimensional kernel. So it must be spanned by $[\gro_N]$.
Since $p_\bfv^*[\gro_N]$ is in $\ker\pi_B^*$ and $(p_\bfv\circ \pi_\bfv)^*[\omega_N] =  l_2 \gamma$, we must have that $\pi_L^* \gamma = 0$.
It follows that $PD(D_1)$ is also in the kernel of $\pi_B^*$ and since it is clearly independent of $p_\bfv^*[\gro_N]$ we conclude that
$p_\bfv^*[\gro_N]$ and $PD(D_1)$ span $\ker\pi_B^*$.
 
The induced K\"ahler class on $B_{l_1,l_2,\bfv,\bfw}=(S_n,\grD)$ is clearly in $\ker\pi_B^*$ and so the lemma follows.
\end{proof}

In view of Lemma \ref{kahclassB} we write the induced K\"ahler class $[\omega_B]$ on $(S_n,\grD)$ as
\begin{equation}\label{BKahform}
[\gro_B]=k_1p_\bfv^*[\gro_N]+k_2 PD(D_1)
\end{equation}

\begin{lemma}\label{l2k2}
The following hold:

\begin{enumerate}
\item $k_2=l_2$,
\item $k_1= m_1 l_1 w_2$
\end{enumerate}
\end{lemma}

\begin{proof}
Since we know that $\pi_\bfv^*[\gro_B]$ is a trivial class in $M_{l_1,l_2,\bfw}$
and $(p_\bfv\circ \pi_\bfv)^*[\omega_N] = l_2 \gamma$ while $\pi_\bfv^*PD(D_1) = -m_1l_1w_2 \gamma$, we see immediately that
$k_1 l_2 - k_2 m_1 l_1 w_2 = 0$ and since $\gcd(k_1,k_2) = m = l_2/s$, we conclude that
$k_2=l_2$ while $k_1 = m_1 l_1 w_2$.
\end{proof}

In the almost regular case this process can be inverted. Given positive integers $n,m,k_1,k_2$ with $m=\gcd(k_1,k_2)$ we can determine the relatively prime positive integers $w_1,w_2$ from the equation
$$\frac{w_2}{w_1}=\frac{k_1}{nk_2+k_1}$$
and the relatively prime positive integers $l_1,l_2$ from
$$\frac{l_1}{l_2}=\frac{n}{m(w_1-w_2)}.$$
This gives an analog of diagram (32) of \cite{BoTo13} together with its interpretation depicted in the diagram
\begin{equation}\label{prodHir}
\begin{matrix} && M_{l_1,l_2,\bfw} && \\
                        &&&&  \\
                        & \pi_\bfw \swarrow  &&\searrow\pi_\bfv & \\
                        &&&& \\
                       & N\times \bbc\bbp(\bfw) &&& (S_n,\grD).
\end{matrix}
\end{equation}
Thus, we can view $M_{l_1,l_2,\bfw}$ in two ways. First, the southwest arrow describes an $S^1$ orbibundle over the K\"ahler orbifold $N\times \bbc\bbp^1[\bfw]$ with its product structure. Second the southeast arrow describes the K\"ahler structure of a $\bbc\bbp^1$-bundle over $N$ with twisted complex structure and a mild orbifold structure on the fibers given as a quotient by an almost regular Reeb vector field. Note that in (32) of \cite{BoTo13} the southeast arrow is the quotient by a regular Reeb vector field.

\section{The Topology of the Joins}

Since we are mainly interested in compact Sasaki manifolds, which have finite fundamental group, we shall assume that the Sasaki manifold $M$ is simply connected. It is then easy to construct examples with cyclic fundamental group. From the homotopy exact sequence of the fibration $S^1\ra{1.5}M\times S^3\ra{1.5} M_{l_1,l_2,\bfw}$ we have
\begin{proposition}\label{simcon}
If $M$ is simply connected, then so is $M_{l_1,l_2,\bfw}$. Moreover, if $M$ is 2-connected, $\pi_2(M_{l_1,l_2,\bfw})\approx \bbz$.
\end{proposition}

We now describe our method for computing the cohomology ring of the join $M_{l_1,l_2,\bfw}$.

\subsection{The Method}
Our approach uses the spectral sequence method employed in \cite{WaZi90,BG00a} (see also Section 7.6.2 of \cite{BG05}). The fibration $\pi_L$ in Diagram (\ref{s2comdia}) together with the torus bundle with total space $M\times S^3_\bfw$ gives the commutative diagram of fibrations
\begin{equation}\label{orbifibrationexactseq}
\begin{matrix}M\times S^3_\bfw &\ra{2.6} &M_{l_1,l_2,\bfw}&\ra{2.6}
&\mathsf{B}S^1 \\
\decdnar{=}&&\decdnar{}&&\decdnar{\psi}\\
M\times S^3_\bfw&\ra{2.6} & N\times\mathsf{B}\bbc\bbp^1[\bfw]&\ra{2.6}
&\mathsf{B}S^1\times \mathsf{B}S^1\, 
\end{matrix} \qquad \qquad
\end{equation}
where $\mathsf{B}G$ is the classifying space of a group $G$ or Haefliger's classifying space \cite{Hae84} of an orbifold if $G$ is an orbifold. Note that the lower fibration is a product of fibrations. In particular,  the fibration 
\begin{equation}\label{cporbfib}
S^3_\bfw \ra{2.6} \mathsf{B}\bbc\bbp^1[\bfw]\ra{2.6} \mathsf{B}S^1
\end{equation}
is rationally equivalent to the Hopf fibration, so over $\bbq$ the only non-vanishing differentials in its Leray-Serre spectral sequence are $d_4(\grb)=s^2$ where $\grb$ is the orientation class of $S^3$ and $s$ is a basis in $H^2( \mathsf{B}S^1,\bbq)\approx \bbq$ and those induced from $d_4$ by naturality. However, we want the cohomology over $\bbz$. 

\begin{lemma}\label{cporbcoh}
For $w_1$ and $w_2$ relatively prime positive integers we have
$$H^r_{orb}(\bbc\bbp^1[\bfw],\bbz)=H^r( \mathsf{B}\bbc\bbp^1[\bfw],\bbz)= \begin{cases}
                    \bbz &\text{for $r=0,2$,}\\                  
                    \bbz_{w_1w_2} &\text{for $r>2$ even,}\\
                     0 &\text{for $r$ odd.}
                     \end{cases}$$           
\end{lemma}

\begin{proof}
As in \cite{BG05} we cover the $\mathsf{B}\bbc\bbp^1[\bfw]$ with two overlapping open sets $p^{-1}(U_i)\approx \tU_i\times_{\grG_i}EO$ where $U_i$ is $\bbc\bbp^1\setminus \{0\}$ and $\bbc\bbp^1\setminus \{\infty\}$ for $i=1,2$, respectively. The Mayer-Vietoris sequence is
$$\ra{.8}H^r(\mathsf{B}\bbc\bbp^1[\bfw],\bbz)\ra{1.0} H^r(p^{-1}(U_1),\bbz)\oplus H^r(p^{-1}(U_2)\ra{1.0} H^r(p^{-1}(U_1)\cap p^{-1}(U_2),\bbz)\ra{.4}\cdots$$
Now $p^{-1}(U_i)\approx \tU_i\times_{\grG_i}EO$ is the Eilenberg-MacLane space $K(\bbz_{w_i},1)$ whose cohomology is the group cohomology
$$H^r(\bbz_{w_i},\bbz)=\begin{cases}
                    \bbz &\text{for $r=0$,}\\
                    \bbz_{w_i} &\text{for $r>0$ even,}\\
                     0 &\text{for $r$ odd.}
\end{cases}$$
Moreover, $p^{-1}(U_1)\cap p^{-1}(U_2)=\tU_1\cap\tU_2\times_{\grG_1\cap\grG_2}EO$ and  since $w_1,w_2$ are relatively prime $\grG_1\cap\grG_2=\bbz_{w_1}\cap\bbz_{w_2}=\{\BOne\}$.
So for $r=2$ the Mayer-Vietoris sequence becomes
\begin{equation}\label{LS2}
0\ra{1.8}\bbz\fract{j}{\ra{1.8}}H^2(\mathsf{B}\bbc\bbp^1[\bfw],\bbz)\ra{1.8}\bbz_{w_1w_2}\ra{1.8} 0.
\end{equation}
From the $E_2$ term of the Leray-Serre spectral sequence of the fibration (\ref{cporbfib}), we see that the map $j$ in (\ref{LS2}) must be multiplication by $w_1w_2$ implying that $H^2(\mathsf{B}\bbc\bbp^1[\bfw],\bbz)\approx \bbz$.

For $r>2$ even the sequence gives $H^r(\mathsf{B}\bbc\bbp^1[\bfw],\bbz)\approx \bbz_{w_1}\oplus \bbz_{w_2}\approx \bbz_{w_1w_2}$, whereas for $r$ odd $H^r(\mathsf{B}\bbc\bbp^1[\bfw],\bbz)\approx 0$.
\end{proof} 

One now easily sees that 

\begin{lemma}\label{LS}
The only non-vanishing differentials in the Leray-Serre spectral sequence of the fibration (\ref{cporbfib}) are those induced naturally by $d_4(\gra)= w_1w_2s^2$ for $s\in H^2(\mathsf{B}S^1,\bbz)\approx \bbz[s]$ and $\gra$ the orientation class of $S^3$.       
\end{lemma}

Now the map $\psi$ of Diagram (\ref{orbifibrationexactseq}) is that induced by the inclusion $e^{i\theta}\mapsto (e^{il_2\theta},e^{-il_1\theta})$. So noting 
$$H^*(\mathsf{B}S^1\times \mathsf{B}S^1,\bbz)=\bbz[s_1,s_2]$$ 
we see that $\psi^*s_1=l_2s$ and $\psi^*s_2=-l_1s$. This together with Lemma \ref{LS} gives $d_4(\gra)=w_1w_2l_1^2s^2$ in the Leray-Serre spectral sequence of the top fibration in Diagram (\ref{orbifibrationexactseq}).

Further analysis depends on the differentials in the spectral sequence of the fibration 
\begin{equation}\label{MNspec}
M\ra{1.5}N\ra{1.5}\mathsf{B}S^1. 
\end{equation}

\begin{algorithm}
Given the differentials in the spectral sequence of the fibration (\ref{MNspec}), one can use the commutative diagram (\ref{orbifibrationexactseq}) to compute the cohomology ring of the join manifold $M_{l_1,l_2,\bfw}$.
\end{algorithm}

It is worth mentioning that the finiteness of deformation types of smooth Fano manifolds implies a bound on the Betti numbers of the join which only depends on dimension. This gives a Betti number bound on the manifolds obtained from our construction when $N$ is Fano. In particular, in dimension seven $b_2(M^7_{l_1,l_2,\bfw})\leq 9$, whereas, in dimension nine we have the bound $b_2(M^9_{l_1,l_2,\bfw})\leq 10$ \cite{BG00a}.

\subsection{An Example in General Dimension}\label{gendex}
One case that is particularly easy to describe in all odd dimensions is when $M$ is the odd-dimensional sphere $S^{2r+1}$ with $r=2,3,\ldots,$. Here we have $N=\bbc\bbp^r$ which is monotone with Fano index $\cali_N=r+1$. We have

\begin{theorem}\label{topcpr}
The join $M^{2r+3}_{l_1,l_2,\bfw}=S^{2r+1}\star_{l_1,l_2}S^3_\bfw$ has integral cohomology ring
given by
$$\bbz[x,y]/(w_1w_2l_1^2x^2,x^{r+1},x^2y,y^2)$$
where $x,y$ are classes of degree $2$ and $2r+1$, respectively.
\end{theorem}

\begin{proof}
The $E_2$ term of the Leray-Serre spectral sequence of the top fibration of diagram (\ref{orbifibrationexactseq}) is 
$$E^{p,q}_2=H^p(\mathsf{B}S^1,H^q(S^{2r+1}\times S^3_\bfw,\bbz))\approx \bbz[s]\otimes\grL[\gra,\grb],$$ 
where $\gra$ is a $3$-class and $\grb$ is a $2r+1$ class. By the Leray-Serre Theorem this converges to $H^{p+q}(M_{l_1,l_2,\bfw}^{2r+3},\bbz)$.
From the usual Hopf fibration and Lemma \ref{LS} the only non-zero differentials in the Leray-Serre spectral sequence of the bottom fibration in Diagram (\ref{orbifibrationexactseq}) are $d_4(\gra)=w_1w_2s^2_2$ and $d_{2r+2}(\grb)=s^{r+1}_1$. By naturality the differentials of the top fibration of (\ref{orbifibrationexactseq}) are $d_4(\gra)=w_1w_2(-l_1s)^2$ and $d_{2r+2}(\grb)=(l_2s)^{r+1}$. It follows that $H^{p}(M_{l_1,l_2,\bfw}^{2r+3},\bbz)$ has an element $x$ of degree $2$ with $w_1w_2l_1^2x^2$ vanishing, and since $l_2$ is relatively prime to $w_1w_2l_1^2$,  $x^p$ vanishes for $p\geq r+1$. Similarly, for dimensional reasons there is an element $y$ of degree $2r+1$ such that $y^2$ and $x^2y$ vanish.
\end{proof}

The connected component $\gA\gu\gt_0(M^{2r+3}_{l_1,l_2,\bfw})$ of the Sasaki automorphism group $\gA\gu\gt(M^{2r+3}_{l_1,l_2,\bfw})$ of the manifolds $M^{2r+3}_{l_1,l_2,\bfw}$ is $SU(r+1)\times T^2$. Hence, these manifolds are toric Sasaki manifolds. In fact, the join of toric Sasaki manifolds is a toric Sasaki manifold \cite{Boy10a}. However, our methods only make essential use of the 2-dimensional $\bfw$-subcone, not the full Sasaki cone. The manifolds $M^{2r+3}_{l_1,l_2,\bfw}$ are studied further in \cite{BoTo14b}.

\section{Admissible CSC constructions}\label{CSCconstruct}

We now pick up the thread from Section \ref{admissible} and describe the construction (see also
\cite{ACGT08}) of admissible K\"ahler metrics on $S_n$ (in fact, more generally on log pairs $(S_n, \Delta)$).
Consider the circle action
on $S_n$ induced by the natural circle action on $L_n$. It extends to a holomorphic
$\mathbb{C}^*$ action. The open and dense set ${S_n}_0\subset S_n$ of stable points with respect to the
latter action has the structure of a principal $\mathbb{C}^*$ bundle over the stable quotient.
The hermitian norm on the fibers induces via a Legendre transform a function
$\gz:{S_n}_0\rightarrow (-1,1)$ whose extension to $S_n$ consists of the critical manifolds
$\gz^{-1}(1)=P(\BOne\oplus 0)$ and $\gz^{-1}(-1)=P(0 \oplus L_n)$.
Letting $\theta$ be a connection one form for the Hermitian metric on ${S_n}_0$, with curvature
$d\theta = \omega_{N_n}$, an admissible K\"ahler metric and form on the base $S_n$ are
given up to scale by the respective formulas
\begin{equation}\label{g}
g=\frac{1+r\gz}{r}g_{N_n}+\frac {d\gz^2}
{\Theta (\gz)}+\Theta (\gz)\theta^2,\quad
\omega = \frac{1+r\gz}{r}\omega_{N_n} + d\gz \wedge
\theta,
\end{equation}
valid on ${S_n}_0$. Here $\Theta$ is a smooth function with domain containing
$(-1,1)$ and $r$, is a real number of the same sign as
$g_{N_n}$ and satisfying $0 < |r| < 1$. The complex structure yielding this
K\"ahler structure is given by the pullback of the base complex structure
along with the requirement $Jd\gz = \Theta \theta$. The function $\gz$ is hamiltonian
with $K= J\,grad\, \gz$ a Killing vector field. In fact, $\gz$ is the moment 
map on $S_n$ for the circle action, decomposing $S_n$ into 
the free orbits ${S_n}_0 = \gz^{-1}((-1,1))$ and the special orbits 
$D_1= \gz^{-1}(1)$ and $D_2=\gz^{-1}(-1)$.
Finally, $\theta$ satisfies $\theta(K)=1$.

\begin{remark}\label{hamiltonian2form}
Note that on $S_n$
$$\phi := \frac{-(1+r \gz)}{r^2} \omega_{N_n} + \gz d\gz \wedge \theta$$ is a Hamiltonian $2$-form of order one.
\end{remark}

\bigskip

We can now interpret $g$ as a metric on the log pair $(S_n,\Delta)$ with
$$\grD= (1-1/m_1)D_1+(1-1/m_2)D_2$$ if
$\Theta$ satisfies the positivity and boundary
conditions
\begin{equation}\label{positivity}
\begin{array}{l}
\Theta(\gz) > 0, \quad -1 < \gz <1,\\
\\
 \Theta(\pm 1) = 0,\\
 \\
 \Theta'(-1) = 2/m_2\quad \quad \Theta'(1)=-2/m_1.
\end{array}
\end{equation}

\begin{remark}
This construction is based on the symplectic viewpoint where different choices of $\Theta$ yields different complex structures all compatible with the same fixed symplectic form $\omega$. However, for each $\Theta$ there is an $S^1$-equivariant diffeomorphism pulling back $J$ to the original fixed complex structure on $S_n$ in such a way that the K\"ahler form of the new K\"ahler metric is in the same cohomology class as $\omega$ \cite{ACGT08}. Therefore, with all else fixed, we may view the set of the functions $\Theta$ satisfying \eqref{positivity} as parametrizing a family of K\"ahler metrics within the same K\"ahler class of $(S_n,\Delta)$.
\end{remark}

\bigskip

The K\"ahler class $\Omega_{\mathbf r} = [\omega]$ of an admissible metric is also called
{\it admissible} and is uniquely determined by the parameter
$r$, once the data associated with $S_n$ (i.e.
$d_N$, $s_{N_n}$, $g_{N_n}$ etc.) is fixed. In fact, up to scale
\begin{equation}\label{admKahclass}
 \Omega_{\mathbf r}  = [\omega_{N_n}]/r + 2 \pi PD[D_1+D_2],
 \end{equation}
where $PD$ denotes the Poincar\'e dual.
The number $r$,
together with the data associated with $S_n$ will be called {\it admissible data}.

Define a function $F(\gz)$ by the formula $\Theta(\gz)=F(\gz)/\gp(\gz)$, where
$\gp(\gz) =(1 + r \gz)^{d_{N}}$.
Since $\gp(\gz)$ is positive for $-1\leq \gz \leq1$, conditions
\eqref{positivity}
are equivalent to the following conditions on $F(\gz)$:
\begin{equation}
\label{positivityF}
\begin{array}{l}
F(\gz) > 0, \quad -1 < \gz <1,\\
\\
F(\pm 1) = 0,\\
\\
F'(- 1) = 2\,\gp(-1)/m_2 \quad \quad F'( 1) =-2\,\gp(1)/m_1.
\end{array}
\end{equation}

\subsection{The CSC condition}

From \cite{ApCaGa06} we have that the scalar curvature of an admissible metric  given by \eqref{g} equals
\begin{equation}\label{scal}
Scal =
\frac{2 d_N s_{N_n} r }{1+r\gz} - \frac{F''(\gz)}{\gp(\gz)}.
\end{equation}
Thus the CSC
condition is equivalent the 
ODE
\begin{equation} \label{CSCode}
F''(\gz) = \left(2 d_N s_{N_n} r - k (1 + r \gz) \right)(1 + r \gz)^{d_{N}-1},
\end{equation}
where $k$ is a constant (equal to $Scal$ when \eqref{CSCode} is solved). Notice that if \eqref{CSCode} has a solution such that the boundary conditions from \eqref{positivityF} holds, then it will also follow that $F(\gz) > 0$ for $-1 < \gz <1$ and thus all of \eqref{positivityF} is satisfied.
To see this, merely observe that since $(1 + r \gz)^{d_{N}-1}>0$ for $0<|r|<1$ and $-1<\gz<1$, then $F''(\gz)$ can change sign at most once over the interval $-1<\gz<1$. Together with this fact, the endpoint conditions rule out any possibility of $F(\gz)$ being zero for any $-1<\gz<1$.

Integrating and using the conditions of $F'(\pm 1)$ in \eqref{positivityF}, we immediately get that
\[ F'(\gz) =  \left(2  s_{N_n}  - k \frac{1}{r(d_{N}+1)} (1 + r \gz)\right)(1 + r \gz)^{d_{N}}+c,\]
where
\begin{equation}\label{whatcis}
c=  \frac{2 \left(1-r^2\right)^{d_{N}} (m_2 (1-r)+m_1 (1+r) -  2m_1 m_2 s_{N_n})}{m_1 m_2 \left((1+r)^{d_{N}+1}-(1-r)^{d_{N}+1}\right)}
\end{equation}
and 
\begin{equation}\label{whatkis}
k= \frac{2 (d_N+1) r \left(m_2 (1+r)^{d_N} (1+m_1 s_{N_n})-m_1 (1-r)^{d_N} (-1+m_2 s_{N_n})\right)}{m_1m_2 \left((1+r)^{d_N+1}-(1-r)^{d_N+1}\right)}.
\end{equation}

Now we have a solution to \eqref{CSCode}, namely
\[
F(\gz) = \int_{-1}^\gz \left(\left(2  s_{N_n}  - k \frac{1}{r(d_{N}+1)} (1 + r t)\right)(1 + r t)^{d_{N}}+c \right)\, d\gz,
\]
satisfying \eqref{positivityF} iff
\begin{equation}\label{CSCfinaleqn}
\int_{-1}^1 \left(\left(2  s_{N_n}  - k \frac{1}{r(d_{N}+1)} (1 + r \gz)\right)(1 + r \gz)^{d_{N}}+c \right)\, d\gz = 0.
\end{equation}

Now we integrate \eqref{CSCfinaleqn} to arrive at the equation

\begin{equation}\label{candkeqn}
\frac{2s_{N_n}\left( (1+r)^{d_N+1} - (1-r)^{d_N+1}\right)}{r(d_N+1)} - \frac{k\left( (1+r)^{d_N+2} - (1-r)^{d_N+2}\right)}{r^2(d_N+1)(d_N+2)} + 2c=0.
\end{equation}
Thus the existence of an admissible CSC K\"ahler metric on the log pair $(S_n,\Delta)$ correspond to solving all three equations
\eqref{whatcis}, \eqref{whatkis}, and \eqref{candkeqn}. 

\subsection{Extremal K\"ahler metrics}
If we generalize equation \eqref{CSCode} to the condition that $Scal$ from \eqref{scal} is a affine function of $\gz$, then we obtain the equation
\begin{equation}\label{extremalode}
F''(\gz) = (1+r \gz)^{d_N-1}(2d_Ns_{N_n} r + (\alpha\gz +\beta)(1+r\gz)),
\end{equation}
where $\alpha$ and $\beta$ are constants.
It is well known that this corresponds to extremal K\"ahler metrics (see e.g. \cite{ACGT08}).
Moreover, similarly to the smooth case, one easily sees (by integrating and solving for $A$ and $B$) that \eqref{extremalode} has a unique solution $F(\gz)$ satisfying the endpoint conditions of \eqref{positivityF}. 
Finally, if the K\"ahler form $\omega_N$ on $N$ is assumed to have positive scalar curvature, this polynomial $F(\gz)$ also satisfies the positivity condition of \eqref{positivityF} by the standard root-counting argument introduced by Hwang \cite{Hwa94} and Guan \cite{Gua95}. For completeness we give the root-counting argument for this special case: Assume for contradiction that the positivity condition of \eqref{positivityF} fails. Then, due to the endpoint conditions on $F(\gz)$, the function $F(\gz)$ has at least two relative maxima and at least one relative minimum 
inside the interval $(-1,1)$. Thus, in the interval $(-1,1)$,  the concavity of $F(\gz)$ changes at least twice, i.e. $F''(\gz)$ is zero at least twice. Since $(1+r \gz)^{d_N-1} >0$ for $-1<\gz<1$, we see that this implies that the second order polynomial \newline
$P(\gz) = (2d_Ns_{N_n} r + (\alpha\gz +\beta)(1+r\gz))$ has two roots inside $(-1,1)$ and further the concavity changes exactly twice. Thus $F(\gz)$  has two relative maxima at $\gz = a_1$ and $\gz = a_3$ and one relative minimum at $\gz=a_2$ such that
$-1<a_1<a_2<a_3<1$ and the roots $r_1,r_2$ of $P(\gz)$ are such that $a_1<r_1<a_2<r_2<a_3$. Moreover,
$P(a_1)<0$ and $P(a_3) <0$. Now we observe that $P(-1/r) = 2d_Ns_{N_n} r$ and thus if $s_{N_n}r \geq 0$, we see that $P(\gz)$ must have one more root in either $[-1/r, a_1)$ (if $r>0$) or $(a_3,-1/r]$ (if $r<0$). Obviously $P(\gz)$ cannot have three roots and so we have a contradiction. We conclude that  the positivity condition of\eqref{positivityF}  must be satisfied.

This yields the following proposition which also proves Theorem \ref{allextr} as we shall see below.

\begin{proposition}\label{extrexistence} Assume that the scalar curvature $s_N$ of $(N,\omega_N)$ is non-negative.
For any log pair $(S_n, \Delta)$, any admissible K\"ahler class on $S_n$ contains an admissible extremal metric which is smooth in the orbifold sense on $(S_n, \Delta)$.
\end{proposition}

\subsection{The Einstein Conditions}
A K\"ahler metric is KE if and only if
$$\rho - \lambda \omega=0$$ for some constant $\lambda$. From \cite{ApCaGa06} we have that the Ricci form of an admissible metric  given by \eqref{g} equals
\begin{equation}\label{rho}
\rho =
\rho_{N} - \frac{1}{2} d d^c \log F =s_{N_n}\omega_{N_n} - 
\frac{1}{2}\frac{F'(\gz)}{\gp(\gz)}  \omega_{N_n}
-\frac{1}{2}\Bigl(\frac{F'(\gz)}{\gp(\gz)}\Bigr)'(\gz) d\gz \wedge \theta.
\end{equation}
Thus the KE 
condition is equivalent the 
ODE
\begin{equation} \label{KEodes}
\frac{F'(\gz)}{\gp(\gz)}  = 2s_{N_n} - 2 \lambda 
(\gz + 1/r).
\end{equation}
Now \eqref{positivityF} implies the necessary conditions
$$ 
\begin{array}{ccl}
s_{N_n} -  \lambda 
(-1 + 1/r)& = & 1/m_2\\
\\
s_{N_n}-  \lambda 
(1 + 1/r)&=& -1/m_1,
\end{array}
$$
which are equivalent to
\begin{equation}\label{fano}
\begin{array}{ccl}
2\lambda& = & 1/m_2+1/m_1\\
\\
2s_{N_n}r &=& (1+r)/m_2 + (1-r)/m_1.
\end{array}
\end{equation}

Since $s_{N_n}r >0$ we see that the base manifold $N$ (not surprisingly) must have positive scalar curvature.
If  \eqref{fano} is satisfied, then
\eqref{KEodes} is equivalent to the ODE:
\begin{equation} \label{KEode}
\frac{F'(\gz)}{\gp(\gz)}  = (1-\gz)/m_2 -(1+\gz)/m_1
\end{equation}

Now it is easy to see that for a solution satisfying \eqref{positivityF} to exist we
need
\begin{equation}\label{KEintegral}
\int_{-1}^1 \left((1-\gz)/m_2 -(1+\gz)/m_1\right) {\gp(\gz)} d\gz = 0.
\end{equation}
On the other hand, if this is satisfied
\begin{equation}\label{KEmetricF}
F(\gz) := \int_{-1}^\gz \left((1-t)/m_2 -(1+t)/m_1\right) {\gp(t)} dt
\end{equation}
would yield a solution of \eqref{KEodes} satisfying all the conditions of \eqref{positivityF}. 
Setting $s_{N_n}=\cali_N/n$ in the second equation of \eqref{fano} we have the following result.
\begin{proposition}\label{KEprop} 
Given admissible data and a choice of $m_1,m_2$ as above the admissible metric  \eqref{g}, 
with $\Theta(\gz) = \frac{F(\gz)}{\gp(t)}$ and $F(\gz)$ given by \eqref{KEmetricF},
is KE iff  
$$2r\cali_N/n = (1+r)/m_2 + (1-r)/m_1$$
and \eqref{KEintegral} are both satisfied.
\end{proposition}

\section{CSC and Extremal Rays}
In order to finish the proofs of Theorem \ref{admjoincsc} and Theorem \ref{allextr}, we now connect the K\"ahler geometry of Section \ref{CSCconstruct} with the Sasaki geometry of Section \ref{thejoin}. Assume $M_{l_1,l_2,\bfw}=M\star_{l_1,l_2}S^3_\bfw$ is the join as described in the beginning of the Section \ref{thejoin} with the induced contact structure $\cald_{l_1,l_2,\bfw}$, and now we assume that $w_1>w_2$. Let $\bfv=(v_1,v_2)$ be a weight vector with relatively prime integer components and let $\xi_\bfv$ be the corresponding Reeb vector field in the $\bfw$-Sasaki cone $\gt^+_\bfw$. Let the log pair $(S_n,\grD)$ be
the quotient of $M_{l_1,l_2,\bfw}$ by the flow of the Reeb vector field $\xi_\bfv$.
Using Theorem \ref{preSE} we have
$m_i=v_i\frac{l_2}{s}$ and
$n=l_1\bigl(\frac{w_1v_2-w_2v_1}{s}\bigr)$, where $s=\gcd(|w_2v_1-w_1v_2|,l_2)$.
Writing $[\omega_{N_n}] = 2\pi n p_\bfv^*[\omega_N]$ and using that \newline
$$PD[D_1+D_2] = 2PD(D_1) - PD(D_1-D_2) = 2PD(D_1) - n[\omega_N],$$ 
we see that \eqref{admKahclass} can be re-written to
$$ \Omega_{\mathbf r}  =4\pi\left(\frac{n (1- r)}{2r} [\omega_{N}] +  PD(D_1)\right)$$
and so $[\omega_B]$ given by Lemma \ref{l2k2} is indeed admissible, where $r$ is such that
$$\frac{n (1- r)}{2r}  = k_1/k_2 = m_1 l_1 w_2/l_2$$
which gives
\begin{equation}\label{eqnr} 
r= \frac{w_1v_2-w_2v_1}{w_1v_2+w_2v_1},
\end{equation}
and 
\begin{equation}\label{admKahclass2}
 \Omega_{\mathbf r}  =4\pi\left(\frac{k_1}{l_2}[\gro_N] +PD(D_1)\right)= \frac{4\pi}{l_2}\left(k_1[\gro_N] +k_2PD(D_1)\right)=\frac{4\pi}{l_2}[\gro_B] .
\end{equation}

For a description of extremal Sasaki metrics we refer the reader to \cite{BGS06} and Section 4.4 of \cite{BoTo13}.

\subsection{Lifting the Admissible Data}\label{adsas}
We now want to lift the admissibility conditions on $(S_n,\grD)$ to $M_{\bfl,\bfw}$ using Theorem \ref{preSE}, and we need to determine the scale factor involved in this lifting. Let $M_0$ denote the dense open subspace of $M_{l_1,l_2,\bfw}$ defined by the condition $z_1z_2\neq 0$, and let $Z_i$ be the submanifolds of $M_{\bfl,\bfw}$ defined by setting $z_{i+1}=0$ with $i=1,2\mod 2$. This gives a stratification 
\begin{equation}\label{MZ}
M_{l_1,l_2,\bfw}  =M_0\sqcup Z_1\sqcup Z_2.
\end{equation}
It is easy to see that 
\begin{lemma}\label{ZD}
For each pair of relatively prime positive integers $v_1,v_2$ the dense open submanifold $M_0$ is the total space of an $S^1$-bundle over $S_{n0}$ and $Z_i=\pi_\bfv^{-1}(D_i)$ is independent of $\bfv$ and $n$.
\end{lemma}

This lemma says that although the quotient spaces of different Reeb vector fields in the Sasaki cone may be quite different even topologically, their lifted geometry on $M_{l_1,l_2,\bfw}$ is similar. 

Now Theorem \ref{preSE} shows that the quotient space of $M_{l_1,l_2,\bfw}$ by the circle action generated by the quasi-regular Reeb vector field $\xi_\bfv$ is a ruled projective algebraic orbifold given as the log pair $(S_n,\grD)$; however, although there is a specific Sasakian structure $\cals_\bfv$ on $M_{l_1,l_2,\bfw}$ the theorem does not specify the K\"ahler structure on $(S_n,\grD)$. It is now our purpose to do so, and relate it to $\cals_\bfv$.

\begin{proposition}\label{screl}
Let $v_1,v_2$ be relatively prime positive integers, and consider the Sasakian structure $\cals_\bfv=(\xi_\bfv,\eta_\bfv,\Phi,g_\bfv)$ on $M_{l_1,l_2,\bfw}$. Then the induced K\"ahler structure $(g_B,\gro_B)$ on $(S_n,\grD)$ satisfies
$$g^T =\frac{l_2}{4\pi}\frac{\pi_\bfv^* g}{mv_1v_2}=\frac{\pi_\bfv^* g_B}{mv_1v_2},\qquad d\eta_\bfv=\frac{l_2}{4\pi} \frac{\pi_\bfv^*\gro}{mv_1v_2} =\frac{\pi_\bfv^*\gro_B}{mv_1v_2}=\gro^T$$
where $g^T=d\eta\circ (\BOne\otimes \Phi)$.
\end{proposition}

\begin{proof}
For any $\xi_\bfv\in \gt_\bfw^+$ define the quadratic form $q_\bfv=v_1|z_1|^2+v_2|z_2|^2$. Then the Sasakian structure $\cals_\bfv$ is related to reducible Sasakian structure $\cals_\bfw$ by $q_\bfv\eta_\bfv=q_\bfw\eta_\bfw$. This gives the relation between the transverse K\"ahler forms,
\begin{equation}\label{dvw}
d\eta_\bfv|_\cald = q_\bfv^{-1}q_\bfw d\eta_\bfw|_\cald.
\end{equation}
Now choose coordinates on $S^3$ so that $q_\bfv=v_1|z_1|^2+v_2|z_2|^2=2\grk$ with $\grk\in \bbr^+$. 
Let $\gtz:M_{l_1,l_2,\bfw}\ra{1.6} [-1,1]$ be the moment map of the lifted circle action of the moment map $\gz$. Then Lemma \ref{ZD} implies 
$$\gtz= \frac{\grk-v_2|z_2|^2}{\grk} = \frac{v_1|z_1|^2 - \grk}{\grk}$$
which gives
$$|z_1|^2 = (\grk\gtz+\grk)/v_1\qquad |z_2|^2 = (\grk-\grk\gtz)/v_2.$$
This gives, using Equation \eqref{eqnr}
\begin{equation}\label{qvqwrel}
q_\bfv^{-1}q_\bfw=\frac{(w_1v_2-w_2v_1)\gtz+w_1v_2+w_2v_1}{2v_1v_2}=\frac{w_1v_2-w_2v_1}{2v_1v_2}(\gtz+r^{-1}).
\end{equation}

Now by Equation \eqref{admKahclass2} the K\"ahler form $(4\pi/l_2)\gro_B$ is in the admissible class $\grO_\bfr$ and we choose it to be admissible. So $(4\pi/l_2)\gro_B=\gro$. Thus, pulling back and using Equation \eqref{g} we have by identifying $N$ with the zero section of $S_n=\bbp(\BOne \oplus L_n)$,
\begin{equation}\label{BN}
\gro_B|_N= \frac{l_2n}{2}(r^{-1}+\gz)\gro_N.
\end{equation}

Thus, to find the scale factor we write 
\begin{equation}\label{sceqn}
\pi_\bfv^*\gro_B=bd\eta_\bfv
\end{equation}
for some positive constant $b$. Now from the commutative Diagram \eqref{t3comdia} we have $\pi_\bfw^*\gro_N=\pi_\bfv^*\gro_N$. To find $b$ it is enough to compare coifficients of the pullback of $\gro_N$ on both sides of Equation  \eqref{sceqn}. Using Equations \eqref{dvw}-\eqref{BN} together with the equation for $n$ in Theorem \ref{preSE} gives $b=mv_1v_2$.
\end{proof}

\begin{remark}\label{coverrem}
The factor $mv_1v_2$ can be thought of as arising from the multiple cover argument in Diagram \ref{actcomdia} which occurs on Sasakian level as well.
\end{remark}

Proposition \ref{screl} allows us to consider the Sasakian structure $\cals_\bfv$ as an {\it admissible Sasakian structure}. We simply view $\Theta$ as a function of the lifted moment map $\gtz$. This function $\Theta(\gtz)$ satisfies the positivity conditions and boundary conditions of Equation \eqref{positivity}. We then get a Sasaki metric in the usual way, namely $g_\bfv=g^T+\eta_\bfv\otimes \eta_\bfv$ together with its full Sasakian structure $\cals_\bfv=(\xi_\bfv,\eta_\bfv,\Phi_\bfv,g_\bfv)$. Although this construction was done for a pair of relatively prime positive integers $v_1,v_2$ we can extend this to the entire ray by applying a transverse homothety $(\xi,\eta)\mapsto (a^{-1}\xi,a\eta)$ which implies the following scaling of the admissible data:
$$\theta\mapsto a^{-1}\theta,\qquad \Theta\mapsto a\Theta,\qquad m_i\mapsto a^{-1}m_i,$$
and $m$ is scale invariant. This defines the Sasaki admissible data for all quasi-regular Reeb vector fields. 

We now wish to extend the concept of admissible Sasaki data to the irregular case. For this we consider the components $v_1,v_2$ of $\bfv$ to be any positive real numbers. We shall assume that the function $\Theta$ of Section \ref{CSCconstruct} is chosen such that $m\Theta$ is independent of $m$ and varies smoothly with $v_1,v_2$. As we shall see later this is the case for any (quasi-regular) extremal Sasakian structure. We need 
\begin{lemma}\label{vfam}
Let $v_1,v_2\in \bbr^+$. Then the family of transverse K\"ahler metrics and forms of Proposition \ref{screl} vary smoothly with $\bfv=(v_1,v_2)$.
\end{lemma}

\begin{proof}
It is convenient to rewrite the transverse metric $g^T$ and K\"ahler form $\gro^T$ of Proposition \ref{screl} on the dense open set $M_0$ in the form
\begin{equation}\label{gT2}
g^T=\frac{l_2}{4\pi}\Bigl(\frac{2\pi l_1(w_1v_2-w_2v_1)}{l_2v_1v_2}(r^{-1}+\gtz)\pi_\bfv^*g_N +\frac{d\gtz^2}{\tilde{\Theta}(\gtz)} +\tilde{\Theta}(\gtz)\ttheta^2\Bigr)
\end{equation}
$$\gro^T=\frac{l_2}{4\pi}\Bigl(\frac{2\pi l_1(w_1v_2-w_2v_1)}{l_2v_1v_2}(r^{-1}+\gtz)\pi_\bfv^*\gro_N +d\gtz\wedge \ttheta\Bigr)$$
where $\tTheta=mv_1v_2\pi_\bfv^*\Theta$ and $\ttheta=\frac{\pi_\bfv^*\theta}{mv_1v_2}$. Note that $\tTheta$ satisfies the boundary conditions $\tTheta(\pm 1)=0,$ $\tTheta'(-1)=2v_1$ and $\tTheta'(1)=-2v_2.$ We will ignore the term $l_2/4\pi$ and consider the terms in brackets as our admissible data.

We claim that  we can interpret Equation \eqref{gT2} as a family of transverse K\"ahler metrics and forms that varies smoothly with $\bfv$. First from the commutative Diagram \eqref{t3comdia} we see that $\pi_\bfv^*g_N=\pi_\bfw^*g_N$, so the term $\pi_\bfv^*g_N$ is independent of $\bfv$.

So to show that the family is smooth on $M_0$ we only need to show $\ttheta$ is a family of 1-forms on $M_0$ that varies smoothly with $\bfv$. Since on $M_0$ we have coordinates induced by $z_1,z_2$ such that 
$$v_1|z_1|^2= \grk+\gre,\qquad v_2|z_2|^2= \grk-\gre$$
where $-\grk<\gre<\grk$. This trivializes $M_0$ as $M_0\approx T^2\times (-\grk,\grk)\times N$ as well as $S_{n0}\approx S^1\times (-\grk,\grk)\times N$. Now the Hamiltonian vector field $K$ vanishes nowhere on $S_{n0}$ and lifts to a vector field on $M_0$. Choosing $\grk=v_1$ we see that this vector field is $H_1$ with moment map $\gtz$ and satisfies $\pi_{\bfv *}H_1=mv_1v_2K$ (cf. Remark \ref{coverrem}).  Since $\ttheta$ is a pullback we have $\ttheta(\xi_\bfv)=0$ implying that $\ttheta(H_2)=-v_1/v_2$. Moreover, since both $H_1$ and $H_2$ are nowhere vanishing on $M_0$ we have coordinates $\varphi_1$ and $\varphi_2$ such that $\ttheta=d\varphi_1 -\frac{v_1}{v_2}d\varphi_2 +A$ where the $A$ is a 1-form on $N$ satisfying 
$$dA=\frac{2\pi l_1(w_1v_2-w_2v_1)}{l_2v_1v_2}\pi_\bfv^*\gro_N.$$ 
Since $\pi_\bfv^*\gro_N$ is independent of $\bfv$, this shows that $\ttheta$ depends smoothly on $\bfv $ on $M_0$.

As in the K\"ahler case the admissible quasi-regular Sasaki structures smoothly extend to the boundary $Z_1\sqcup Z_2$ with the indicated boundary conditions. Moreover, any irregular Sasakian structure $\cals_\bfv$ in the $\bfw$-Sasaki cone can be represented as a limit of quasi-regular structures by Theorem 7.1.10 of \cite{BG05} which from the above can be taken to be admissible. Hence, by continuity the irregular admissible structures on $M_0$ extend to the boundary as well.
\end{proof}

\begin{remark}\label{isorem}
Beginning from a K\"ahler class $\grO_r$ of Equation \eqref{admKahclass} we obtain an admissible K\"ahler form within the K\"ahler class by performing a deformation of the form $\gro\mapsto \gro +i\partial\bar{\partial}\varphi$ where the function $\varphi$ is invariant under the Hamiltonian circle action. This is equivalent on the Sasaki level to a deformation of the contact structure of the form $\eta\mapsto \eta +\zeta$ where $\zeta$ is a basic 1-form that is invariant under the lifted Hamiltonian circle action. Once this is done for a fixed $\bfv$ we see from the discussion above that it holds for all $\bfv$.
\end{remark}

\begin{remark}\label{rayremark}
It is convenient to consider the space of rays of the $\bfw$-Sasaki cone. 
We let $\gR_\bfw$ denote the {\it space of rays} in $\gt_\bfw^+$ and the ray defined by the vector $\bfv$ by $\bar{\bfv}$. By mapping a ray $\bar{\bfv}\in\gt_\bfw^+$ to its slope $v_2/v_1$ gives a homeomorphism of $\gR_\bfw$ with the open interval $(0,\infty)$. It follows from Equation 7.3.12 of \cite{BG05} that under the transverse homothety $(\xi,\eta)\mapsto (a^{-1}\xi,a\eta)$ extremality as well as constant scalar curvature are preserved. Thus, being extremal or CSC is a property of rays and descends to $\gR_\bfw$. Let $\gR^{\gr\ga\gt}_\bfw$ denote the subset of {\it rational rays}, that is, those rays with rational slope. By Theorem 7.1.10 of \cite{BG05}, $\gR^{\gr\ga\gt}_\bfw$ is dense in $\gR_\bfw$. Moreover, for every rational ray there is a unique pair of relatively prime positive integers $v_1,v_2$. So by Theorem \ref{preSE} there is unique log pair $(S_n,\grD)$ associated to the ray $\gr\in\gR^{\gr\ga\gt}_\bfw$.
\end{remark}
 
\subsection{Applying the Admissible Sasaki Data}

For a choice of co-prime integers $(v_1,v_2) \neq (w_1,w_2)$ and the values of $m_i$, $n$, and $r$ given as above, we recall that the metric \eqref{g} is extremal when $\Theta(\gz)$, satisfying the boundary conditions \eqref{positivity}, is such that when $\Theta(\gz)=F(\gz)/\gp(\gz)$, then 
$F(\gz)$ satisfies the ODE \eqref{extremalode}. The constants $\alpha$ and $\beta$ are uniquely determined from this ODE and the boundary conditions.

Now we are setting $s_{N_n} = A/n = \frac{ A s}{l_1(w_1v_2-w_2v_1)}$, where, by Remark \ref{sNn}, $A \leq d_N+1$. In any case, $A$ depends solely on $(N,g_N, \omega_N)$.  (If $\omega_N$ is K\"ahler-Einstein, $A$ is just $\cali_N$ as introduced in Remark \ref{sNn}). As a consequence, since $m=l_2/s$, $m s_{N_n}$ depends only on the join data and the choice of $(v_1,v_2)$. Thus the function $m\Theta(\gz)$ is independent of $m$ and varies smoothly with $v_1,v_2$. This is precisely the assumption we need to be able to use Lemma \ref{vfam}, and so moving forward any pair $(v_1,v_2)$ such that $v_1,v_2\in \bbr^+$ has a well-defined ``extremal'' $\tilde{\Theta}(\tilde{\gz})$ resulting in the existence of an admissible extremal Sasakian metric whenever 
$\tilde{\Theta}(\tilde{\gz})>0$ for $-1<\tilde{\gz} <1$. 

Notice that together with Propostion \ref{extrexistence}, this proves that when the scalar curvature $s_N$ of $(N,\omega_N)$ is non-negative, then each ray in $\gR_\bfw$ can be represented by extremal Sasaki metrics. Consequently this proves Theorem \ref{allextr}. 

Assuming $w_1>w_2$, the existence of an admissible quasi-regular CSC ray in the $\bfw$-Sasaki cone $\gt^+_\bfw$ corresponds to showing that for a choice of $(v_1,v_2) \neq (w_1,w_2)$ and the values of $m_i$, $n$, and $r$ given as above, the equation system 
\eqref{whatcis}, \eqref{whatkis}, and \eqref{candkeqn} is solved. 
Notice that with $s_{N_n} = A/n = \frac{ A s}{l_1(w_1v_2-w_2v_1)}$ the value $s$ (equivalently $m$) predictably cancels from the equation system 
\eqref{whatcis}, \eqref{whatkis}, and \eqref{candkeqn}.  In fact, we have (assuming $(v_1,v_2) \neq (w_1,w_2)$), that the system is equivalent to the equation $f(b)=0$ for $b>0$, where
$b = \frac{v_2}{v_1} \neq \frac{w_2}{w_1}$ is the slope alluded to in Remark \ref{rayremark} and

\begin{equation}\label{functionf}
\begin{array}{ccl}
f(b) & = &  w_1^{2(d_N+1)} b^{2 d_N+3}( A l_2 + l_1 (d_N + 1)  w_2-b (d_N + 1 ) l_1 w_1 )\\
\\
& - & w_1^{d_N+2}  w_2^{d_N} b^{d_N + 3}  ((d_N+1)  (A ( d_N+1) l_2 - l_1 (( d_N+1) w_1 + ( d_N+2) w_2)))\\
\\
& + & w_1^{d_N+1}  w_2^{d_N+1}  b^{d_N + 2}  (2 A d_N (d_N+2) l_2 - (d_N+1)(2d_N+3) l_1 (w_1 + w_2))\\
\\
& - &  w_1^{d_N}  w_2^{d_N+2}  b^{d_N + 1}(d_N+1) (A ( d_N+1) l_2 - l_1 ((d_N+2) w_1 + ( d_N+1) w_2))\\
\\
& + & w_2^{2 (d_N + 1)}( b (A l_2 + l_1 (d_N + 1) w_1)-( d_N + 1 ) l_1 w_2).
\end{array}
\end{equation}

When a solution $b\in \bbq^+$, we have a quasi-regular CSC metric and, since CSC is just a special case of extremal, when $b\in \bbr^+ \setminus \bbq^+$, we have an irregular CSC metric.

Since $f(b)$ is a polynomial which is formally defined for any real value of $b$ and
$$f(\frac{w_2}{w_1})=f'(\frac{w_2}{w_1})=f''(\frac{w_2}{w_1})=0$$ while
$$f '''(\frac{w_2}{w_1}) =  3(d_N+1)(d_N+2) l_1 w_1^{d_N} w_2^{d_N} (w_1 - w_2)> 0$$ and
$$\lim_{b\rightarrow + \infty} f(b) = - \infty,$$
we see that there is at least one solution $b \in (\frac{w_2}{w_1},+\infty)$ to $f(b) =0$.
This completes the proof of Theorem \ref{admjoincsc}.  \hfill$\Box$

In the special case when $\omega_N$ is K\"ahler-Einstein, $A=\cali_N >0$ and
$$l_1=\frac{\cali_N}{\gcd(w_1+w_2, \cali_N)}, \qquad l_2 = \frac{w_1+w_2}{\gcd(w_1+w_2, \cali_N)}$$ 
we set $t=w_2/w_1$ and realize that the equation $f(b)=0$ (and $b \neq t$) above reduces to the equation $h(b)=0$ where
$$h(b) = (1+d_N)b^{d_N+3} - (1+2t+d_N t) b^{d_N+2} + (2+d_N+t)t^{d_N+1} b - (1+d_N) t^{d_N+2}.$$
We easily check that $h(t)=h\,'(t)=0$, i.e. $h$ has a double root at the forbidden $b=t$. Moreover, 
$$h''(t) <0,\quad \text{and}\quad \lim_{b\rightarrow +\infty}h(b) = + \infty,$$
confirming what we already know from above, namely that there is at least one CSC ray. In this case, however,  we can also use Descartes' rule of signs to see that counting with multiplicity there are at most three positive roots of $g(b)$ and so there is at most one admissible CSC ray in the $\bfw$-Sasaki cone.

To finish the proof of the first part of Theorem \ref{admjoinse} we need to show that the $\bfw$-Sasaki cone has a Reeb vector field giving an admissible Sasaki-Einstein structure. Obviously this must then correspond to the one and only admissible CSC ray in the $\bfw$-Sasaki cone. The proof of the last statement of Theorem \ref{admjoinse} is given in Section \ref{srssec} below.

For co-prime integers $(v_1,v_2) \neq (w_1,w_2)$ and the values of $m_i$, $n$, and $r$ given as above, we realize that
Proposition \ref{KEprop} implies that the admissible extremal Sasaki structure associated to
$\xi_\bfv$ is  $\eta$-Einstein (and thus, up to transverse homothety, SE)
iff
\begin{equation}\label{KEintegral2}
\int_{-1}^1 \left((v_1-v_2)-(v_1+v_2)\gz) \right)((w_1v_2+w_2v_1)+ (w_1v_2-w_2v_1)\gz)^{d_N}d\gz = 0.
\end{equation}
Again, we set $w_2/w_1=t$ and $v_2/v_1=b$ and assume $b\neq t$. We also still assume $0<t< 1$ (i.e. $w_1>w_2$). Now equation \eqref{KEintegral2} is equivalent to
\begin{equation}\label{KEintegral3}
\int_{-1}^1 \left((1-b)-(1+b)\gz) \right)((b+t)+ (b-t)\gz)^{d_N}d\gz = 0.
\end{equation}
Let $j(b)$ denote the left hand side of \eqref{KEintegral3} and assume $t\in (0,1)\cap{\mathbb Q}$ is fixed. Now it is easy to check that
$$j(t) >0\quad\quad\text{and}\quad\quad \lim_{b\rightarrow +\infty}j(b) =-\infty.$$
Thus $\exists~ b \in (t,+\infty)$ such that \eqref{KEintegral3} is solved. 
This completes the proof of the first part of Theorem \ref{admjoinse}.  \hfill$\Box$

Although the majority of the SE structures obtained in this paper are irregular,
we can, however, produce many quasi-regular SE cases as follows: Set $b=kt$. Then 
\eqref{KEintegral3} is equivalent with
\begin{equation}\label{KEintegral4}
\int_{-1}^1 \left((1-kt)-(1+kt)\gz) \right)((k+1)+ (k-1)\gz)^{d_N}d\gz = 0.
\end{equation}
or
\begin{equation}\label{KEintegral5}
t= \frac{\int_{-1}^1 \left(1-\gz \right)((k+1)+ (k-1)\gz)^{d_N}d\gz}{k\int_{-1}^1 \left(1+\gz \right)((k+1)+ (k-1)\gz)^{d_N}d\gz}.
\end{equation}
\begin{lemma}
For $k>1$,
$$0<\int_{-1}^1 \left(1-\gz \right)((k+1)+ (k-1)\gz)^{d_N}d\gz<\int_{-1}^1 \left(1+\gz \right)((k+1)+ (k-1)\gz)^{d_N}d\gz.$$
\end{lemma}
\begin{proof}
The first inequality is obvious and the next is equivalent to 
$$\int_{-1}^{1} \gz ((k+1)+ (k-1)\gz)^{d_N}d\gz >0.$$
By integrating, this in turn is equivalent to
$$-d_{N_n} + (2+d_{N_n})k - (2+d_{N_n})k^{d_{N_n}+1} + d_{N_n}k^{d_{N_n}+2} >0.$$
Setting $p(k) = -d_{N_n} + (2+d_{N_n})k - (2+d_{N_n})k^{d_{N_n}+1} + d_{N_n}k^{d_{N_n}+2}$ we  observe that $p(1)=p'(1)=0$ while $p\,''(k) > 0$ for all $k>1$. Thus $p(k)>0$ for all $k>1$ and hence the inequality holds.
 \end{proof}
Now it follows that for any given $k\in (1,+\infty)\cap {\mathbb Q}$, $\exists~ t \in (0,1)\cap {\mathbb Q}$ (determined by \eqref{KEintegral5}) such that if the co-prime integers $w_1$ and $w_2$ are such that $w_2/w_1=t$ and then co-prime integers $v_1$ and $v_2$ are such that $v_2/v_1= kt$ (and $l_1$ and $l_2$ are chosen according to Lemma \ref{c10}) then the ray determined by $(v_1,v_2)$ in the $\bfw$-Sasaki cone contains a quasi-regular SE structure.

\begin{example}\label{Ypq} $Y^{p,q}$. These is an infinite sequence of toric contact structures on $S^2\times S^3$ that admit an SE metric in their Sasaki cone discovered by the physicists \cite{GMSW04a}. The pair $(p,q)$ consists of relatively prime positive integers satisfying $1\leq q<p$. This was treated in Example 4.7 of \cite{BoPa10} although the conventions\footnote{In particular, there we chose $w_1\leq w_2$; whereas, here we use the opposite convention, $w_1\geq w_2$.} are slightly different. Here we set
\begin{equation}\label{pqw}
\bfw=\frac{1}{\gcd(p+q,p-q)}\bigl(p+q,p-q\bigr).
\end{equation}
Note that the conditions on $p,q$ eliminate the case $\bfw=(1,1)$. It is also easy to see that 
\begin{equation}\label{pqrel}
l_1=\gcd(p+q,p-q),\qquad l_2=p.
\end{equation}
Now we have $N=S^2$ so $\cali_N=2$ and $d_N=1$. Most of the SE metrics on $Y^{p,q}$ are irregular; however, one can obtain quasi-regular solutions from Equation \eqref{KEintegral5} which simplifies to
$$t=\frac{2+k}{k(1+2k)}.$$
The quasi-regular SE solutions are given by the Diophantine equation \cite{GMSW04a}
\begin{equation}\label{qrYpq}
4p^2-3q^2=n^2, 
\end{equation}
for some $n \in \bbz$. Note that the Equation \eqref{qrYpq} has an infinite number of solutions. A particular example is $(p,q)=(13,8)$ which in terms of our join parameters gives $\bfw=(21,5),l_1=1,l_2=13$. The Reeb vector field corresponding to the SE metric is determined by $\bfv=(5,7)$ which gives the the base orbifold Hirzebruch surface $(S_{122},\grD)$ with
$$\grD= (1-\frac{1}{65})D_1 + (1-\frac{1}{91})D_2$$
and $m=13$. Of course, this base orbifold has a positive KE orbifold metric.

In general, the following result follows from Theorem \ref{admjoinse}.
\begin{proposition}\label{ypqse}
The Reeb vector field of the unique Sasaki-Einstein metric of $Y^{p,q}$ lies in the $\bfw$-Sasaki cone with $\bfw$ determined by Equations (\ref{pqw}).
\end{proposition}

\end{example}

As the examples in the next subsection will illustrate we can also produce many
non-Einstein quasi-regular CSC rays in all dimensions.

\subsection{A Special Case: $N= \bbc\bbp^{r}$}\label{specialsec}
Consider the special case of Section \ref{gendex} in which case $A=\cali_N = r+1$. Now if we let $(l_1,l_2)$ be any relatively prime pair of positive integers except $(\frac{r+1}{\gcd(|\bfw|, r+1)}, \frac{|\bfw|}{\gcd(|\bfw|, r+1)})$ we know that the CSC ray from the proof of Theorem \ref{admjoincsc} is not  Sasaki-Einstein by Lemma \ref{c10}, and by Equation \eqref{c1cald2} $c_1(\cald_{l_1,l_2,\bfw})\neq 0$. 
Again, for the majority of choices of $(w_1,w_2)$, the CSC ray discovered will be irregular. However we can produce quite
a lot of quasi-regular CSC rays as the example below shows. 

This case is studied in much more depth in \cite{BoTo14b}. In particular, it is shown there that if $r,l_1$ and $\bfw$ are fixed, there is only a finite number of diffeomorphism types among the manifolds $M_{l_1,l_2,\bfw}$. So for each $r>1,l_1,w_1,w_2$ with $w_1>w_2$, there exists a smooth $2r+3$-dimensional manifold $M_{l_1,\bfw}$ which admits a countably infinite number of contact structures of Sasaki type each with a compatible Sasaki metric of constant scalar curvature.

\begin{example}
For example with $A=r+1$, $l_1=\frac{r}{\gcd(|\bfw|, r)}$, and $l_2 = \frac{|\bfw|}{\gcd(|\bfw|, r)}$ we set $t=\frac{w_2}{w_1}$ and $b=k t$ in $f(b)$ above. Then the equation $f(b)=0$ is equivalent to 
$$t= \frac{r -(r + 1)k + k^{r+1}}{k(1-(r+1)k^{r} + r k^{r+1})}.$$
It is a straightforward calculus exercise to check that for $k>1$  we get a solution $0<t<1$ as predicted in the previous section
(so $b= k t > t =  \frac{w_2}{w_1}$ and $w_1 > w_2$). In particular, if we pick a rational $k>1$  we get a rational $t$. This value of $t$ will determine $(w_1,w_2)$ and then with $(v_1,v_2)$ such that $v_2/v_1=b = k t$ we have our CSC quasi-regular Sasaki metric.
\end{example}

\begin{example}
Let us assume that $r=2$ (hence $A=3$), $l_1=1$ and $\bfw=(3,2)$. So to have smooth 7-manifolds $M^7_{1,l_2,(3,2)}$ we must have $\gcd(l_2,6)=1$. Then $f(b) = 3(2-3b)^3 g(b)$, where 
$$g(b) = 81 b^5 - 27 (l_2-4) b^4 - 54 (l_2-2) b^3 + 36 (l_2-1) b^2 + 8 (l_2-6) b - 16.$$
Now $g(2/3) = -32<0$ and $\displaystyle \lim_{b-> + \infty}g(b) = + \infty$, justifying the solution to $g(b)=0$ in the interval $(2/3, +\infty)$ as already established.
Notice however that $g(0) = -16<0$ and $g(1/3) = (13 l_2-115)/3$. So for any $l_2 \geq 9$ with $\gcd(l_2,6)=1$, we have two additional solutions to $g(b)=0$ in the interval $(0,2/3)$. Furthermore, one can check that the other two solutions are negative, so there are 3 rays of CSC Sasaki metrics in the $\bfw$-cone. It can also be checked that for $l_2=1,5,7$, there is only one solution to $g(b)=0$. We thus have

\begin{proposition}\label{3csc}
For each $l_2\geq 9$ relatively prime to $6$ there are three distinct constant scalar curvature rays in the $\bfw$-Sasaki cone of the toric contact 7-manifold $(M^7_{1,l_2,(3,2)},\cald_{1,l_2,(3,2)})$.
\end{proposition}

It also follows from our results in \cite{BoTo14b} that infinitely many of the manifolds $M^7_{1,l_2,(3,2)}$ are diffeomorphic. Thus, there exists an infinite subsequence $s_j\subset \{l_2\}$ of the integers $l_2\geq 9$ giving distinct contact structures $\cald_{s_j}$ of Sasaki type occurring on the same 7-manifold all containing three rays of CSC Sasaki metrics in their $\bfw$-Sasaki cone.
\end{example}

\begin{example}\label{w1=w2} Wang-Ziller manifolds.
In the calculus analysis we have done on $f(b)$ so far, we have assumed that $w_1>w_2$. For arguments sake let us assume that $w_1=w_2=1$ in which case our manifolds $M_{l_1,l_2,(1,1)}=M^{1,r}_{l_2,l_1}$, a Wang-Ziller manifold \cite{WaZi90}. If we assume that $N=\bbc\bbp^2$ and pick $l_1=1$, we know from Proposition 2.3 of Wang and Ziller that $M^{1,2}_{l_2,1}$ is $S^2\times S^5$ when $l_2$ is even and the non-trivial $S^5$-bundle over $S^2$, which we denote by $S^2\tilde{\times}S^5$, when $l_2$ is odd. So there are exactly these two diffeomorphism types. Moreover, we know that we have at least one CSC ray, namely the regular ray in the $S^1$-bundle over the product $N \times \bbc\bbp^1$. This case corresponds to $b=1$, although $f(b)$ has no geometric meaning for $b=1$.  However, we also get that
$f(b) = -3 ( b-1)^4g(b)$ with
$$g(b)=  (1 + (3-l_2) b - 4 b^2 l_2 + 6 b^2 + (3-l_2) b^3 + b^4).$$
Now we observe that $g(0) = 1>0$, $g(1) = 2(7-3l_2)$, and $\displaystyle \lim_{b\rightarrow +\infty} g(b) = +\infty$. So for $l_2\geq 3$ $f(b)$ has at least 2 roots not equal to $1$ ; one in the interval $(0,1)$ and one in the interval $(1,+\infty)$. Thus we have at least three CSC rays in this case as well.

Now the Wang-Ziller manifolds are toric, in fact, they are homogeneous, and in our case $M^{1,2}_{l_2,1}$ have a four-dimensional Sasaki cone, and when $l_2\geq 3$ the $\bfw$-Sasaki cone (i.e. the 2-dimensional Sasaki cone associated with $S^3$) has three CSC rays, one regular and the other two irregular or quasi-regular. Notice also in our case the first Chern class of the contact bundle is $c_1(\cald_{1,l_2})=(3l_2-2)\grg$. Summarizing we have

\begin{theorem}\label{WZex}
The 7-manifolds $S^2\times S^5$ and $S^2\tilde{\times}S^5$ admit countably infinite inequivalent toric contact structures $\cald_{1,l_2}$ of Reeb type with $l_2$ even for the former and $l_2$ odd for the latter. Furthermore, when $l_2\geq 3$ these contact structures admit three distinct rays of Sasaki metrics with constant scalar curvature in $\gt^+_\bfw$.
\end{theorem}

As $l_2$ varies the contact structures are clearly inequivalent as contact structures, not just as toric contact structures. 

\begin{remark}\label{prem}
In this Wang-Ziller case two of the three CSC metrics are actually equivalent under a transformation in the Weyl group $\bbz_2$ acting on the unreduced $\bfw$-cone $\gt^+_\bfw$. This transformation sends a root to its reciprocal. Thus, there are only two CSC Sasaki metrics in the moduli space $\grk$. See the proof of Theorem 1.3 in \cite{Leg10} for another approach to this phenomenon. 
\end{remark}

\end{example}

\subsection{Multiple CSC Rays}
The multiple CSC rays in Proposition \ref{3csc} and Theorem \ref{WZex} illustrate a somewhat common phenomenon that was first illustrated in the case of quadrilateral toric structures ($S^3$-bundles over $S^2$) by Legendre \cite{Leg10}. Consider $f(b)$ in \eqref{functionf}. 
As already stated, any positive solution $b \neq \frac{w_2}{w_1}$ to the equation $f(b)=0$ corresponds to a CSC ray in the 
$\bfw$-Sasaki cone. So far we know that, assuming $w_1>w_2$, there is at least one solution in the interval $(\frac{w_2}{w_1}, +\infty)$.
Since $$f(\frac{w_2}{w_1})=f'(\frac{w_2}{w_1})=f''(\frac{w_2}{w_1})=0$$ while
$$f '''(\frac{w_2}{w_1}) =  3(d_N+1)(d_N+2) l_1 w_1^{d_N} w_2^{d_N} (w_1 - w_2)> 0$$
we know that for $b<\frac{w_2}{w_1}$ sufficiently close to $\frac{w_2}{w_1}$, we have $f(b)<0$. Further it is easy to see that $f(0)<0$. Now we notice that
$f(\frac{w_2}{2w_1})$ is a linear function of $l_2$ with slope equal to
$$\frac{Aw_2^{2d_N+3}}{2^{2d+3}w_1}\left[1+2^{d_N}\left(2^{d_N+2} - (d_N^2 + 2 d_N +5) \right)\right].$$
When $A>0$, which is equivalent to the scalar curvature $s_N$ of $(N,\omega_N)$ being positive, then this slope is positive and thus for sufficiently large value of $l_2$ we have that $f(\frac{w_2}{2w_1})>0$. In that case we have
at least two more roots; one in the interval $(0,\frac{w_2}{2w_1})$ and one in the interval $(\frac{w_2}{2w_1},\frac{w_2}{w_1})$.  As Example \ref{w1=w2} illustrates, even if $w_1=w_2=1$, we can have several CSC rays in the $\bfw$-Sasaki cone. This proves Theorem \ref{3cscthm}.

\subsection{Sasaki-Ricci solitons}\label{srssec}
In this section we finish the proof of Theorem \ref{admjoinse} by proving the existence of Sasaki-Ricci solitons for each ray in the $\bfw$-Sasaki cone. Our definition of Sasaki-Ricci soliton is a slight generalization of the definition found in \cite{FOW06}). In effect, we view the Sasaki structures of the entire ray to be Sasaki-Ricci solitons whenever there is a choice of Reeb vector field in the ray that would give a Sasaki-Ricci solution according to the definition in \cite{FOW06} where the constant $\grl$ in Definition \ref{srsdef} below is fixed. Note that by Proposition 2.2 of \cite{CFO07} the Lie algebra of holomorphic Hamiltonian vector fields defined in \cite{FOW06} coincides with the Lie algebra of transverse holomorphic vector fields. We mention also that Sasaki-Ricci solitons on toric 5-manifolds were studied in \cite{LeTo13}.
\begin{definition}\label{srsdef}
A Sasaki Ricci Soliton (SRS) is a transverse K\"ahler Ricci soliton, that is the equation
$$\rho^T -\lambda \omega^T = \pounds_V\omega^T$$
holds where  $V$ is some transverse holomorphic vector field, and $\lambda$ is some constant.
\end{definition}

We are interested here in shrinking Sasaki Ricci solitons, that is, when the constant $\grl$ is positive.
Now to prove the last statement of Theorem \ref{admjoinse}, we realize that
generalizing \eqref{KEodes} to
\begin{equation} \label{KRSode}
\frac{F'(\gz)}{\gp(\gz)}   - a \frac{F'(\gz)}{\gp(\gz)}= 2s_{N_n} - 2 \lambda 
(\gz + 1/r),
\end{equation}
where $a\in\bbr$ is some constant, corresponds to generalizing
the KE equation $\rho -\lambda \omega=0$ to the K\"ahler Ricci soliton (KRS) equation
$$\rho -\lambda \omega = \pounds_V\omega,$$
with $V = \frac{a}{2} grad_g \gz$.
By following e.g. Section 3 in \cite{ACGT08b} and adapting it to our more general endpoint conditions \eqref{positivityF} (but letting $d_0=d_\infty=0$), it is now straightforward and completely standard
to verify that Proposition \ref{KEprop} generalizes with ``KE'' replaced by ``KRS'',
\eqref{KEintegral} replaced by
\begin{equation}\label{KRSintegral}
\int_{-1}^1  e^{-a\, \gz}\left((1-\gz)/m_2 -(1+\gz)/m_1\right) {\gp(\gz)} d\gz = 0,
\end{equation}
and \eqref{KEmetricF} replaced by
\begin{equation}\label{KRSmetricF}
F(\gz) :=e^{a\,\gz} \int_{-1}^\gz e^{-a\,t} \left((1-t)/m_2 -(1+t)/m_1\right) {\gp(t)} dt.
\end{equation}

It is not hard to show that equation \eqref{KRSintegral} can always be solved for some $a\in \bbr$. This ``$a$'' varies smoothly with $v_1$ and $v_2$ and moreover $mF(\gz)$, hence $m\Theta(\gz)$, is clearly independent of $m$ and varies smoothly with $v_1$ and $v_2$. Finally, $grad_g \gz = -JK$ and $K$ is a Hamiltonian Killing vector field that lifts to the Sasaki manifold as we saw in the proof of Lemma \ref{vfam}.\footnote{See also this proof for the scaling factor between the admissible metric and the resulting transverse metric. The reciprocal of this scaling factor applies to the lift of $V$ above and thus it is easy to see that the resulting vector field (which is basically just a multiple of $H_1$) on the Sasaki manifold depends smoothly on $v_1$ and $v_2$ as well. (see also Lemma 7.1 in \cite{BoTo11}).} Thus we realize that when $\omega_N$ is K\"ahler-Einstein, $A=\cali_N >0$ and the pair $(l_1,l_2)$ is given by Lemma \ref{c10}, the Sasaki structure associated to every single ray, $\xi_\bfv$, in our $\bfw$-Sasaki cone is (up to isotopy) a Sasaki-Ricci soliton.
This proves the last statement in Theorem \ref{admjoinse}.\hfill$\Box$

Our set-up, starting from a join construction, allows for cases where no regular ray in the $\bfw$-Sasaki cone exists.
If, however, the given $\bfw$-Sasaki cone does admit a regular ray, then the transverse K\"ahler structure is a smooth K\"ahler Ricci soliton and the existence of an SE metric in some  ray of the Sasaki cone is predicted by the work of \cite{MaNa13}.

\def\cprime{$'$} \def\cprime{$'$} \def\cprime{$'$} \def\cprime{$'$}
  \def\cprime{$'$} \def\cprime{$'$} \def\cprime{$'$} \def\cprime{$'$}
  \def\cdprime{$''$} \def\cprime{$'$} \def\cprime{$'$} \def\cprime{$'$}
  \def\cprime{$'$}
\providecommand{\bysame}{\leavevmode\hbox to3em{\hrulefill}\thinspace}
\providecommand{\MR}{\relax\ifhmode\unskip\space\fi MR }
\providecommand{\MRhref}[2]{%
  \href{http://www.ams.org/mathscinet-getitem?mr=#1}{#2}
}
\providecommand{\href}[2]{#2}

\end{document}